\documentclass{article}
\usepackage[utf8]{inputenc}
\usepackage{igor-commands}
\usepackage[capitalise]{cleveref}
\usepackage{booktabs}
\usepackage{multirow}
\usepackage[normalem]{ulem}
\usepackage{authblk}

\newtheorem{theorem}{Theorem}[section]
\newtheorem{lemma}[theorem]{Lemma}
\newtheorem{definition}[theorem]{Definition}
\newtheorem{proposition}[theorem]{Proposition}
\newtheorem{corollary}[theorem]{Corollary}

\newtheorem{problem}{Problem}
\theoremstyle{definition}
\newtheorem{example}[theorem]{Example}
\newtheorem{remark}[theorem]{Remark}

\newcommand{\G}{\Gamma}
\newcommand{\La}{\Lambda}
\newcommand{\la}{\lambda}
\newcommand{\Gab}{G^{\textrm{Ab}}}
\newcommand{\Ab}{\textrm{Ab}}
\newcommand{\Kn}{\textrm{Kn}}

\newcommand{\Comm}{\text{Comm}_b}
\newcommand{\x}{\times}

\DeclareMathOperator{\Tup}{Tup}
\DeclareMathOperator{\KnV}{KnV}

\DeclareMathOperator{\Cr}{Cr}
\DeclareMathOperator{\Bg}{B\Gamma}
\newcommand{\Rg}{\Z^{C_\G}}

\newcommand{\bbe}{\vec e}
\newcommand{\bb}{\vec b}

\newcommand{\bx}{\vec x}
\newcommand{\by}{\vec y}

\newcommand{\bv}{\vec v}

\newcommand{\bprod}{\; \square \;}

\title{Graph Powers of Groups II: The RA Matrix}
\author{Gabe Cunningham, Igor Minevich}
\affil{Wentworth Institute of Technology}
\date{\today}

\begin{document}

\maketitle

\begin{abstract}
For a graph $\G$ and group $G$, $G^\G$ is the subgroup of $G^{|\G|}$ generated by elements with $g$ in the coordinates corresponding to $v$ and its neighbors in $\G$. There is a natural epimorphism $G^\G \to (G/[G,G])^\G$ with kernel $[G,G]^n \cap G^\G$. When $[G,G]^n \leq G^\G$, the structure of $G^\G$ is easily described from $(G/[G,G])^\G$. Fixing $\G$, if $[G,G]^{|\G|} \leq G^\G$ for all $G$, we say that $\G$ is RA (reducible to abelian). We showed in \cite{FirstPaper} that wide classes of graphs are RA, including graphs of girth 5 or more. The key tool is the \emph{RA matrix} $C_{\G}$, and we showed that $\G$ is RA if and only if the row space $\Row(C_\G) = \Z^{|\G|}$.

Here, we study the possibilities for the elementary divisors of $C_\G$; the more nontrivial elementary divisors we get, the further $\G$ is from being RA (and the harder $G^\G$ is to describe). We show that while many graphs, including those of girth 4, cartesian products, and most tensor products have at most one nontrivial elementary divisor, one can construct a graph of girth 3 with any prescribed set of elementary divisors and $\Z$-nullity.
\end{abstract}

\section{Introduction}

Given a graph $\G$ and a group $G$, the graph power $G^\G$ is the subgroup of $G^{|\G|}$ generated by elements $g^v$ which have $g$ in coordinates corresponding to $v$ \textit{and} vertices adjacent to $v$ in $\G$, and the identity element $1$ in all other coordinates. This is a natural generalization of the parallel product of group actions (see \cite{parallel-product}) and a vast generalization of the original Lights Out Puzzle by Tiger Electronics, where the group $G$ was $\Z/2\Z$ and the graph was fixed as the $5 \times 5$ grid graph. The short exact sequence

\[1 \to [G, G]^{|\G|} \cap G^\G \to G^\G \to (G/[G, G])^{\G} \to 1\]

shows us that discovering the structure of $G^\G$ really amounts to two things: (1) understanding the subgroup $[G, G]^{|\G|} \cap G^\G$ and (2) doing some linear algebra over $\Z$ and $\Z/k\Z$ to compute the structure of $(\Gab)^\G$ for many graphs $\G$, as done by many authors \cite{FirstPaper, electrician, LOFiniteGraphs,  Eriksson2001note,  Fleischer2013survey, Giffen2009generalizing, Kreh2017lights, MUETZE20072755, Sutner2000sigma, Torrence} in exploring the Lights Out Puzzle over an abelian group (and primarily over $\Z/2\Z$, where it can be shown the all-ones vector is always in $(\Z/2\Z)^{\G}$ for any $\G$ \cite{Sutner1989linear}).

In \cite{FirstPaper} we defined $\G$ to be \emph{$G$-RA} (``reducible to abelian'') if $[G, G]^{|\G|} \sub G^\G$, that is, if, by a series of clicks $g^v$ we can place any commutator $[g, h] \in G$ on any single vertex, leaving the rest of the vertices' states unchanged. If $\G$ is $G$-RA for all groups $G$, we simply say $\G$ is RA. If $\G$ is $G$-RA, understanding the structure of $G^\G$ reduces to the abelian question of understanding $(G/[G, G])^{\G}$, which has already been widely investigated. In \cite{FirstPaper} we illustrated the importance of considering the chain of subgroups
\[\Comm(G, \G) \le [G^\G, G^\G] \le [G, G]^{|\G|} \cap G^{\G} \le [G, G]^{|\G|}\]
where $\Comm(G, \G)$ is generated by commutators $[g^v, h^w]$ for $g,h \in G$ and vertices $v, w$ in $\G$ (possibly $v = w$).

For many graphs, we have proven even the smallest subgroup $\Comm(G, \G)$ gives all of $[G, G]^{|\G|}$ for every $G$, making these RA, namely:
\begin{enumerate}
    \item Any graph with girth 4 or more either (a) with a vertex of degree $\le 2$ or (b) that has edges $u - v - w$ that do not complete to a 4-cycle $u - v - w - x - u$. This includes:
    \begin{enumerate}
        \item $m \times n$ grid graphs with $m, n \ge 2$ and path graphs $P_n$ with $n \ge 3$, 
        \item Cylindrical grids (the cartesian product of a cycle $C_m$ with a path $P_n$ with $n$ vertices) with $m \geq 4$ and $n \geq 3$, and
        \item Torus grids (the cartesian product of $C_m$ with $C_n$) with $m,n \ge 4$. 
        \item All graphs with girth 5 or more other than $K_2$.
    \end{enumerate}
    \item Complete bipartite graphs $K_{m, n}$.
    \item Cube graphs $Q_d$ with $d$ even.
\end{enumerate}

On the other hand, in \cite{FirstPaper} we have shown that cube graphs $Q_d$ with $d$ odd are not RA, and neither are the folded cubes $\square_d$ with $d$ odd. In \cref{eg:CylTori}, we expand on \#1 above and show that for any product of path and cycle graphs (so any higher-dimensional version of grids, cylinders, and tori) the only examples that are not RA are the cubes $Q_d$ with odd $d$. 

The main tool used to tell us all this is what we call the ``RA Matrix'' $C_\G$ for a graph $\G$. We start by building the so-called ``activation matrix'' $A_\G$ for a graph $\G$ by simply adding the identity matrix to its adjacency matrix; the rows (or columns) of $A_\G$ give the generators $g^v$ when the $1$'s in the rows are replaced by $g$ and $0$'s by the identity $1 \in G$. 
Then $C_\G$ is the matrix built by adding ``intersections'' of rows from $A_\G$ to the bottom of $A_\G$, where the intersection of two rows is their bitwise AND; this comes from taking $[g^u, h^v]$ for vertices $u$ and $v$ and noticing that the result of this element of $[G, G]^{|\G|}$ is that the state of all vertices in the intersection of the closed neighborhoods $N[u]$ and $N[v]$ of $u$ and $v$ respectively is multiplied by the commutator $[g, h]$ (by ``closed neighborhood'' $N[v]$ of $v$ we mean all vertices adjacent to $v$ and $v$ itself). One of the main results in \cite{FirstPaper} was that if $\G$ is $G$-RA for every Heisenberg group, then $\G$ is RA. 
This fact, combined with the fact that a Heisenberg group $G$ has central commutator subgroup $[G, G]$, implies that the $\Z$-span of the rows of $C_\G$ tells us everything we need to know about whether $\G$ is RA or not. If the rows of $C_\G$ span all of $\Z^{|\G|}$, we know that we can change the state of any one vertex by $[g, h]$ for any elements $g, h \in G$, for any group $G$. On the other hand, if $C_\G$ has any nontrivial elementary divisors in its Smith Normal Form
\footnote{We always ignore the rectangular nature of the matrix and think of just the $|\G|$ elementary divisors along its main diagonal.},
and $p$ is a prime dividing one of the nontrivial elementary divisors (which may equal 0, in which case any prime $p$ would do), then for the Heisenberg group $G = H(\F_p)$, $\G$ is not $G$-RA. In fact, if there are $k$ elementary divisors of $C_\G$ that are divisible by $p$, then $[H(\F_p), H(\F_p)]^{|\G|} \cap H(\F_p)^\G$ is isomorphic to $(\Z/p\Z)^{|\G| - k}$. In general, the more nontrivial elementary divisors $C_\G$ has, the harder it is to analyze the puzzle. For example, it becomes more difficult to make a concrete theory to determine whether an element of $G^{|\G|}$ is in $G^\G$, or given an element shown by a computer to be in $G^\G$, for a human to find a set of ``clicks'' $g_1^{v_1}, \dots, g_m^{v_m}$ that would result in this element. 

In this paper, we give more criteria for certain graphs to be RA and explore graphs that are not, as well as how far graphs can be from being RA, and we essentially show they can be as far from RA as one could ask for in \cref{thm:allyourbase}. 

We say a graph $\G$ is $1/\mu$-RA if $C_\G$ has a single nontrivial, nonzero elementary divisor and that divisor is $\mu$ (allowing the possibility of $\G$ being RA, in which case $\mu = 1$). In \cref{sec:almost-RA}, we first develop the tools we will use throughout the paper to both show that a graph \textit{is} $1/\mu$-RA and to determine $\mu$ as the gcd of certain numbers related to the degrees of vertices and intersections of neighborhoods in the graph.

In \cref{sec:prods} we explore strong products, cartesian (box) products, and tensor products of graphs. The strong product of two graphs produces the tensor product of their activation matrices, so they are easiest to understand. Cartesian products have particularly nice characterizations: they are $1/\mu$-RA, where $\mu$ is described combinatorially in terms of the properties of the graphs involved in \cref{thm:bprod-1/mu}. This description, and in particular \cref{eg:CylTori}, allows us to extend to the non-abelian setting much work that has been done in the Lights Out Puzzle community related to grids, tori, and cylinders, as well as puzzles where clicking a square affects its entire row and column ($K_m \bprod K_n$), the higher-dimensional analogues of all of these, and others. Tensor products also behave well enough to be characterized as $1/\mu$-RA for a well-described $\mu$ in many cases (see \cref{thm:tensor-bipartite}, \cref{thm:tensor-2bipartite}, and \cref{prop:non-bip-tensor}), and particularly the tensor product of arbitrarily many complete graphs (\cref{thm:tensor-completes}). 

Tensor products allow us to exhibit a family of graphs that are almost RA, namely the crown graphs $\Cr(2n) = K_2 \times K_n$, in \cref{sec:girth-4}, where we continue the exploration of girth 4 graphs started in \cite{FirstPaper} and paint a fairly complete picture of them. A connected, neighborhood-distinguishable\footnote{This means the closed neighborhoods of any two distinct vertices must not be equal.} girth 4 graph $\G$ must be $1/\mu$-RA, where $\mu$ is described in \cref{thm:girth-4-finite-index}. We demonstrate a family of graphs that are $1/\mu$-RA for any $\mu$, namely the crown graphs on $2\mu + 4$ vertices, and show that these are in fact the smallest girth-4 graphs that are $1/\mu$-RA (\cref{thm:crown-graph-smallest}). We also use these graphs and tensor products to construct, for any given $\mu$, an infinite family of graphs that is $1/\mu$-RA in \cref{cor:xK_{k+2}}.

In \cref{sec:girth-3}, a vastly different story about girth 3 graphs unfolds. While most girth 3 graphs (at least on small numbers of vertices) are RA and we can still build a family of $1/\mu$-RA graphs that are girth 3 for each $\mu$ (by just taking the pyramid over crown graphs --- see \cref{prop:pyramid}), we also find a family of graphs $\G$ for which $C_\G$ has arbitrarily many nontrivial nonzero elementary divisors (the Kneser graphs - see \cref{thm:kneser-very-not-RA} and \cref{thm:span-of-CVs}), as well as a family of graphs $\G$ that have arbitrarily large $\Z$-nullity (see \cref{thm:z(n)} and \cref{cor:unbounded-zero-el-div}). The latter is not a family that is known in the House of Graphs \cite{HouseOfGraphs} and may still be unexplored; we call these ``binary graphs'' because their activation matrix is written down using binary representations of the row numbers. We end the section with \cref{thm:allyourbase}, which says that we can make a graph $\G$ for which $\Rg$  has whatever nontrivial elementary divisors we prescribe, both zero and nonzero, essentially confirming that $\G$ can be as far from RA as desired. This paves the way for puzzle makers to explore puzzles over such graphs as they exhibit non-abelian behavior in the largest sense possible so that the order of operations makes the biggest difference.

Most graphs (at least on a small number of vertices) are RA, and we show the relevant data we have gathered in \cref{sec:data}. That means most families of standard graphs do not exhibit interesting non-RA examples, but one family - the Kneser graphs - do give us something to explore, and somehow, in each section of the paper there is something interesting to say about these graphs. 

\section{Notation and Background}

\subsection{Graphs}

We will assume all the graphs in this paper to be finite 
simple graphs, so that each edge is uniquely determined by its endpoints. We will refer to many standard graphs, such as the cycle graph on $n$ vertices $C_n$, the complete graph on $n$ vertices $K_n$, and the complete bipartite graph $K_{m,n}$. We will use $P_n$ to refer to the path graph with $n$ \emph{vertices}. 

We will also encounter the following graphs:
\begin{enumerate}
    \item The cube graph $Q_n$ on $2^n$ vertices, labeled with $n$-bit strings, and with vertices adjacent whenever their labels differ in a single bit.
    \item The crown graph $\Cr(2n)$ on $2n$ vertices; the vertices may be labeled $\{1, 2, \dots, n, 1', 2', \dots, n'\}$, where $i$ is adjacent to $j'$ and $i'$ to $j$ precisely when $i \ne j$. This can also be understood as the tensor product $K_2 \times K_n$ (see \cref{subsec:tensor_prod}). 
    \item The Kneser graph $\Kn(n, k)$ on $k$-sized subsets of $\{1, 2, \dots, n\}$, where two vertices are adjacent if their sets are disjoint.
\end{enumerate}

We use $N(v)$ to denote the neighborhood of $v$ (the vertices adjacent to $v$). We more commonly will want to refer to the closed neighborhood $N[v] = N(v) \cup \{v\}$. As shorthand, when we have a product of graphs, we write $N[u, v]$ for $N[(u, v)]$. We write $d(u,v)$ to mean the distance between $u$ and $v$, i.e., the number of edges in a shortest path between $u$ and $v$.

Let $\G$ be a graph on $n$ vertices $\{v_1, \ldots, v_n\}$. If $v = v_i$ is a vertex of $\G$, then $\bbe_v \in \Z^n$ is defined to be the standard basis vector $\bbe_i$. If $S$ is subset of the vertices, then we define
\[ \vec S = \sum_{v \in S} \bbe_v, \]
so that $\vec S$ is essentially the bit vector that represents $S$. We will most commonly use this notation for $\vec N[v]$. If $\bx$ and $\by$ are vectors whose components are only $0$ and $1$, then we will use $\bx \cap \by$ to mean the bitwise AND of $\bx$ and $\by$. Thus, $\vec N[u] \cap \vec N[v]$ means the same as $\vec S$ where $S = N[u] \cap N[v]$.

\subsection{$G^\G$ and the RA property}

The definition of $G^\G$ is inspired by the metaphor of playing Lights Out on $\G$, but where the state of each vertex is given by an element of $G$. Let us make this precise. Suppose that $\G$ is a graph on $n$ vertices. For any vector $\bx = \langle x_1, \ldots, x_n \rangle\in \Z^n$, we define $g^{\bx}$ to be the element $(g^{x_1}, \ldots, g^{x_n}) \in G^n$.  
Then for each vertex $v \in \G$, we define $g^v$ to be $g^{\vec N[v]}$. In other words, $g^v$ is the element of $G^n$ that is $g$ in the coordinates corresponding to $v$ and its neighbors, and is the identity of $G$ everywhere else. Then we define
\[ G^\G = \langle g^v \mid g \in G, v \in \G \rangle. \]
If we imagine that $\G$ starts out with the identity of $G$ at every vertex, and that ``clicking $v$ with $g$'' multiplies $v$ and its neighbors on the right by $g$, then $G^\G$ is precisely the state space of this puzzle.

If $G$ is abelian, then the structure of $G^\G$ is in some sense straightforward to obtain from linear algebra. When $G$ is non-abelian, however, the analysis of $G^\G$ is more difficult. To help understand $G^\G$, we can consider the natural epimorphism to $(G/[G,G])^\G$, giving rise to the short exact sequence
\[1 \to [G, G]^{|\G|} \cap G^\G \to G^\G \to (G/[G, G])^{\G} \to 1.\]
The easiest possibility to understand is when $[G,G]^{|\G|} \leq G^\G$, which says that given any element of $G^\G$, we may multiply each coordinate independently by any product of commutators and the result will still lie in $G^\G$. In this case, the description of $G^\G$ essentially only depends on the structure of the abelian group $(G/[G,G])^\G$. We say that $\G$ is $G$-RA (short for ``reducible to abelian'') if $[G,G]^{|\G|} \leq G^\G$, and $\G$ is RA if it is $G$-RA for each group $G$.

Note that if $g,h \in G$ and $u,v \in \G$, then $[g^u, h^v] = [g,h]^{\vec N[u] \cap \vec N[v]}$. For many graphs, it turns out that these commutators already generate all of $[G,G]^{|\G|}$, making $\G$ RA. In general, to help understand the subgroup of $[G,G]^{|\G|} \cap G^\G$ that is generated by commutators $[g^u, h^v]$, we define the RA matrix $C_\G$, whose rows are $\vec N[v]$ and $\vec N[u] \cap \vec N[v]$ over all vertices $u$ and $v$. We showed in \cite[Theorem 6.11]{FirstPaper} that $\G$ is RA if and only if $C_\G$ has $n$ elementary divisors of 1 -- in other words, if the integer row space (denoted $\Rg$) is all of $\Z^n$. Furthermore, if $C_\G$ has an elementary divisor divisible by $p$, then it is not $H(\F_p)$-RA, where $H(\F_p)$ is the Heisenberg group of upper-triangular $3 \times 3$ integer matrices modulo $p$.

If $\G$ has connected components $\G_1, \ldots, \G_k$, then $G^\G \cong G^{\G_1} \times \cdots \times G^{\G_k}$, so that we fully understand $G^\G$ just from knowing each $G^{\G_i}$. Hence we will always assume that the graphs we start with are connected, and the only time we will deal with a disconnected graph is when we take the tensor product of two connected bipartite graphs; see \cref{thm:tensor-2bipartite}. 

There is also a natural definition of $G^M$ for an integer matrix $M$: the group generated by $g^{\vec x}$ where $\vec x$ ranges over the rows of $M$. If $M$ and $M'$ are row-equivalent, then their rows generate the same integer lattice, so $G^M = G^{M'}$. In other words, we may as well put $M$ into Hermite form. It is also clear that if $M''$ is obtained from $M$ by permuting the columns (but performing no other column operations), then $G^M \cong G^{M''}$. Thus, we may in principle permute the columns of $M$ as we wish, and then reduce it to Hermite form. However, it turns out that the pivots of the resulting matrix are not an invariant of $M$; they may depend on how we permute the columns. For example, the matrix $\mat2102$ is already in Hermite form, with pivots $[2,2]$. If we switch the columns and then put the new matrix in Hermite form, we get $\mat1204$, and so now the pivots are $[1,4]$. 

\section{Neighborly and almost RA graphs}
\label{sec:almost-RA}

Throughout this paper, we will see that many classes of graphs $\G$ have an RA Matrix $C_\G$ with a single nontrivial elementary divisor. By the theory developed in \cite{FirstPaper}, if the nontrivial elementary divisor is a prime $p$, then the index of $[G^\G, G^\G]$ in $[G, G]^{|\G|}$ is precisely $p$ when $G = H(\F_p)$, the Heisenberg group of order $p^3$. (Indeed, in this case, $[G,G] = \Z/p\Z$ and $[G^\G, G^\G] \cong (\Z/p\Z)^{|\G|-1}$.) In fact, we will be able to say something more about all these classes of graphs: their Hermite form only has at most one element $k \ne 1$ along the diagonal, no matter the rearrangement of the columns of the RA matrix, and $k \ne 0$. 

\begin{definition}
\label{defn:almostRA}
For an integer $k \ge 1$, we will say the graph $\G$ is $1/k$-RA if the diagonal of the Hermite form of $C_\G$ has elements $[1^{|\G|-1}, k]$ for every arrangement of the columns of $C_\G$. We will call $\G$ \emph{almost RA} if it is $1/k$-RA for some $k \ge 1$.
\end{definition}

Note that, by a slight abuse of English, every RA graph is almost RA (with $k = 1$); it will be more convenient to allow this than to rule it out.

\begin{remark}
    If $\G$ is $1/k$-RA, then $k$ is the only nontrivial elementary divisor of $C_\G$ and the determinant of the largest minor of $C_\G$. 
\end{remark}

The following property is clear.

\begin{proposition} \label{prop:1/mu-alt-char}
If $\G$ is $1/k$-RA, then $k$ is the smallest positive integer such that $k\bbe_v \in \Rg$ for any vertex $v$. Furthermore, if $a\bbe_v \in \Rg$, then $k \mid a$.  
\end{proposition}

We will see later that some broad classes of graphs all share the following property which makes it simple to show that they are almost RA.

\begin{definition}
\label{defn:neighborly}
We will call a pair of vertices $\{u,v\}$ \emph{positive} (resp. \emph{negative}) for a graph $\G$ if $\bbe_u + \bbe_v \in \Rg$ (resp. $\bbe_u - \bbe_v \in \Rg$). A pair is \emph{signed} 
if it is positive or negative. A graph $\G$ is \emph{neighborly} if every \textbf{edge} is signed. (That is, every pair $\{u,v\}$ with $u$ and $v$ adjacent.) We say that $\G$ is \emph{positively} (resp. \emph{negatively}) \emph{neighborly} if every edge is positive (resp. negative) for $\G$. 
\end{definition}

\begin{proposition} \label{prop:c-on-pair}
A connected graph $\G$ is neighborly (resp. negatively neighborly) if and only if every pair is signed (resp. negative). Furthermore, if $\G$ is positively neighborly, then whenever $u$ and $v$ are connected by a path of length $m$, $\bbe_u + (-1)^{m-1} \bbe_v \in \Rg$.
\end{proposition}

\begin{proof}
The first part follows easily from the connectedness of the graph; for example, if there is a path $u - v - w$ in the graph and $\bbe_u + \bbe_v$ and $\bbe_v - \bbe_w$ are both in $\Rg$, then so is $(\bbe_u + \bbe_v) - (\bbe_v - \bbe_w) = \bbe_u + \bbe_w$. 

For the second part, consider a path of length $m$ from $u$ to $v$ , say $u = u_0, u_1, \ldots, u_m = v$. 
Then
\[ \bbe_{u} + (-1)^{m-1}\bbe_v = (\bbe_{u_0} + \bbe_{u_1}) - (\bbe_{u_1} + \bbe_{u_2}) + \cdots + (-1)^{m-1} (\bbe_{u_{m-1}} + \bbe_{u_m}) \in \Rg. \qedhere \] 
\end{proof}

    \begin{proposition} \label{prop:nbrly-is-almost-ra}
    If $\G$ is neighborly and $a \bbe_v \in \Rg$ for some $a > 0$ and some vertex $v$, then $\G$ is $1/k$-RA for some divisor $k$ of $a$.
    \end{proposition}

    \begin{proof}
    Let $w$ be the vertex corresponding to the last column of $C_\G$. Since $\G$ is neighborly, for each vertex $u \neq w$, one of the vectors $\bbe_u \pm \bbe_w \in \Rg$. It follows that we can use row operations to reduce the matrix $C_\G$ to an echelon form where the diagonal has a 1 in every column but the last. Since $a\bbe_v \in \Rg$, we also have $a\bbe_v - a(\bbe_v \pm \bbe_w) = \mp a\bbe_w \in \Rg$ so the last number in the diagonal must be nonzero, and a divisor of $a$ (if $a\bbe_w \in \Rg$ and $b\bbe_w \in \Rg$, then $\gcd(a, b)\bbe_w \in \Rg$ by the Euclidean Algorithm). 
    \end{proof}

    In light of \cref{prop:nbrly-is-almost-ra}, we would like to find an easy way to establish that some $a \bbe_v$ is in $\Rg$. The following technical lemma will usually suffice by taking $S = N[v]$ or $S = N[u] \cap N[v]$.

    \begin{lemma} \label{lem:nbrly-kev}
    Suppose $\G$ is neighborly. Let $S$ be a set of vertices such that $\vec S \in \Rg$. Fix $v \in S$, and suppose that $S = \{v\} \sqcup A \sqcup B$ such that $\bbe_v + \bbe_u \in \Rg$ for $u \in A$ and $\bbe_v - \bbe_w \in \Rg$ for $w \in B$. Then $a \bbe_v \in \Rg$, where $a = 1 - |A| + |B|$.
    \end{lemma}

    \begin{proof}
    This follows from
    \[ (1 - |A| + |B|) \bbe_v = \vec S - \sum_{u \in A} (\bbe_{v} + \bbe_{u}) + \sum_{w \in B} (\bbe_v - \bbe_{w}) \in \Rg. \qedhere \]
    \end{proof}

    Combining \cref{prop:nbrly-is-almost-ra} with \cref{lem:nbrly-kev}, we see that a neighborly graph $\G$ will be almost RA unless, in \cref{lem:nbrly-kev}, every $1 - |A| + |B|$ is 0. This puts strong restrictions on the closed neighborhoods ($S = N[v]$) and their pairwise intersections ($S = N[u] \cap N[v])$, and we have not yet found such a graph.

    Now we introduce the main tool that will show us many classes of graphs are particularly nice, being either $1/2$-RA or RA (see \cref{cor:pos-nbr-1/mu,thm:bprod-1/mu,cor:girth-4,kn-prods-ra,thm:K_mxnot-bipartite,thm:tensor-completes,thm:girth-4-finite-index}).

    \begin{theorem} \label{thm:half-ra}
    If $\G$ is neighborly, then the following are equivalent.
    \begin{enumerate}[(1)]
        \item For some $u$ and $v$, both vectors $\bbe_u \pm \bbe_v$ are in $\Rg$.
        \item $2 \bbe_u \in \Rg$ for some vertex $u$.
        \item $\G$ is $1/2$-RA or RA.
    \end{enumerate}
    Furthermore, in this case, $\G$ is $1/2$-RA if and only if every vertex has odd degree and every pair of vertices has an even number of common neighbors (possibly zero).
    \end{theorem}

    \begin{proof}
    (1) clearly implies (2) since $(\bbe_u + \bbe_v) + (\bbe_u - \bbe_v) = 2 \bbe_u$. Conversely, if $2 \bbe_u \in \Rg$, then $\bbe_u + \bbe_v = 2 \bbe_u - (\bbe_u - \bbe_v)$, and so (2) implies (1). Finally, in light of \cref{prop:1/mu-alt-char} and \cref{prop:nbrly-is-almost-ra}, (2) and (3) are equivalent.

    Now, suppose that $\G$ is $1/2$-RA or RA. If every vertex has odd degree and every pair of vertices has an even number of common neighbors, then every row of $C_\G$ has an even sum, and so $2$ must divide the largest elementary divisor of $\G$, which means the determinant of the largest minor is divisible by 2, hence the last entry in the diagonal of the Hermite form must be 2, i.e. $\G$ is $1/2$-RA. Conversely, if any row had an odd sum $k$, then by fixing $u$ and adding vectors of the form $\bbe_u - \bbe_v$ and subtracting vectors of the form $\bbe_u + \bbe_v$, where $v$ ranges over all other vertices whose coordinates are 1 in that row, we would get $c \bbe_u \in \Rg$ for some odd number $c$. Combined with the fact that $2 \bbe_u \in \Rg$, we see that $\bbe_u \in \Rg$, which means $\G$ is RA.
    \end{proof}

    We will use the following fact in the proof of \cref{thm:nbrly-mu}, where we could assume the graph $\G$ is neighborly and the matrix $A = C_\G$ has only a single nontrivial elementary divisor, making the proof easier. However, for the sake of generality and completeness, we prove a much more general result.
    
    \begin{proposition}\label{prop:U-V}
    Let $A$ be an integer matrix and $A\bx \equiv \vec 0 \pmod m$ for some integer $m > 1$ and nonzero vector $\bx$. Then $A$ has an elementary divisor divisible by $m$. 
    \end{proposition}

    \begin{proof}
    Let\footnote{Argument adapted from a solution generated by OpenAI's ChatGPT, August 24, 2025.} $PAQ = D$, with $P$ and $Q$ unimodular and $D$ the diagonal matrix representing the Smith normal form of $A$, and the elementary divisors are $d_1 \mid d_2 \mid \cdots \mid d_n$. Then $\pmod m$ we have $DQ^{-1}\bx = PA\bx = P\vec 0 = \vec 0$. Letting $p$ be a prime divisor of $m$ and $p^k$ the highest power of $p$ dividing $m$, and further letting $x_i$ be the $i$-th entry of $Q^{-1}\bx$, we must have $d_ix_i \equiv 0 \pmod{p^k}$ for all $i$. If there is no elementary divisor $d_i$ that is divisible by $p^k$, then $x_i$ must be divisible by 
    $p$ for each $i$. But that means 
    $\bx$ is in the right kernel of $Q^{-1}$ (mod $p$), which has determinant $\pm 1$ - a contradiction. Thus, there must be an elementary divisor $d_i$ divisible by $p^k$ and in particular $d_n$ is divisible by $p^k$. Such is true for all prime powers dividing $m$, so $d_n$ must be divisible by $m$ (possibly $d_n = 0$).
    \end{proof}

    Note that \cite[Proposition 2.1]{FirstPaper} is the special case of this statement where the $\bx = \vec 1$ and $m$ is prime. 

    Now we can establish that many neighborly graphs are almost RA, and we can even determine the value of $k$ such that they are $1/k$-RA.

    \begin{theorem} \label{thm:nbrly-mu}
    Let $\G$ be neighborly. Suppose that $\G = U \sqcup V$ is a partition of the graph such that $\bbe_u - \bbe_v \in \Rg$ if $u$ and $v$ are in the same part, and otherwise $\bbe_u + \bbe_v \in \Rg$. Let
    \[ \delta = \gcd(\{|N[v] \cap U| - |N[v] \cap V| : v \in \G \}), \textrm{ and} \]
    \[ \kappa = \gcd(\{|N[u] \cap N[v] \cap U| - |N[u] \cap N[v] \cap V| : u,v \in \G, u \neq v \}).\]
    If $\delta$ and $\kappa$ are not both zero, then $\G$ is $1/\mu$-RA, where $\mu = \gcd(\delta, \kappa)$.    
    \end{theorem}

    \begin{proof}
    Let us assume that $v \in U$ and let $S = N[v]$. Let $A = S \cap V$ and $B = (S \setminus \{v\}) \cap U$, so that $S = \{v\} \sqcup A \sqcup B$. Then by \cref{lem:nbrly-kev}, $(1 - |A| + |B|) \bbe_v \in \Rg$. On the other hand, $1 - |A| + |B| = 1 - |S \cap V| + (|S \cap U| - 1) = |S \cap U| - |S \cap V|$. So $d \bbe_v \in \Rg$, with $d = |S \cap U| - |S \cap V|$. A similar calculation shows that the same is true if $v \in V$ instead. It follows that $\delta \bbe_v \in \Rg$.

    Similarly, if $S = N[u] \cap N[v]$ and $w \in S$, then we can take $A = S \cap V$ and $B = (S \setminus \{w\}) \cap U$, and we again find that $k \bbe_v \in \Rg$, where $k = |N[u] \cap N[v] \cap U| - |N[u] \cap N[v] \cap V|$. So $\kappa \bbe_v \in \Rg$.

    By \cref{prop:nbrly-is-almost-ra} and \cref{prop:1/mu-alt-char}, if $\delta$ and $\kappa$ are both nonzero, then $\G$ is $1/\mu$-RA for some $\mu$ dividing $\gamma = \gcd(\delta, \kappa)$. On the other hand, our definition of $\delta$ and $\kappa$ ensures that $|S \cap U| - |S \cap V|$ is a multiple of $\gamma$ for every $S = N[v]$ and $S = N[u] \cap N[v]$. That means for every row of $C_\G$, the sum of the coordinates corresponding to vertices in $U$, less the sum of the coordinates corresponding to vertices in $V$, is divisible by $\gamma$. In other words, $\bx = \sum_{u \in U} \bbe_u - \sum_{v \in V} \bbe_v$ is in the right kernel of $C_\G \pmod \gamma$, so by \cref{prop:U-V} there is an elementary divisor divisible by $\gamma$. Since the only nontrivial elementary divisor is $\mu$, it follows that $\mu = \gamma$.
    \end{proof}

    \begin{corollary} \label{cor:pos-nbr-1/mu}
    Suppose that $\G$ is positively neighborly.
    \begin{enumerate}
        \item If $\G$ is not bipartite, then $\G$ is $1/2$-RA or RA, and it is $1/2$-RA if and only if every vertex has odd degree and every pair of vertices has an even number of common neighbors (possibly zero).
        \item If $\G$ is bipartite, then it is $1/\mu$-RA, where $\mu = \gcd(\delta, \kappa)$ with
    \[ \delta = \gcd(\{\deg(v)-1 : v \in \G\}) \textrm{ and}\]
    \[ \kappa = \gcd(\{|N[u] \cap N[v]| : u,v \in \G, u \neq v, d(u,v) = 2\}). \]
    \end{enumerate}
    \end{corollary}

    \begin{proof}
    First, suppose that $\G$ is not bipartite. Then given any two vertices $u$ and $v$, there is a path of odd length and a path of even length between them. Then since $\G$ is positively neighborly, \cref{prop:c-on-pair} implies that both vectors $\bbe_u \pm \bbe_v$ are in $\Rg$, and the result follows from \cref{thm:half-ra}.

    Now suppose that $\G$ is bipartite with bipartition $U \sqcup V$. Then \cref{prop:c-on-pair} implies that $U$ and $V$ satisfy the conditions of \cref{thm:nbrly-mu}. If $v \in V$, then 
    \[ |N[v] \cap U| - |N[v] \cap V| = \deg(v) - 1, \]
    and otherwise $|N[v] \cap U| - |N[v] \cap V| = 1 - \deg(v)$. If $u$ and $v$ are neighbors, then one of them is in $U$ and one is in $V$, so that $|N[u] \cap N[v] \cap U| - |N[u] \cap N[v] \cap V| = 0$. 
    If $u$ and $v$ are are two apart, then $N[u] \cap N[v]$ is contained in either $U$ or $V$. Then the result follows directly from \cref{thm:nbrly-mu}. 
    \end{proof}

\begin{corollary} \label{thm:neg-nbr-1/mu}
Suppose that $\G$ is negatively neighborly. Then $\G$ is $1/\mu$-RA, where $\mu = \gcd(\delta, \kappa)$ with
    \[ \delta = \gcd(\{\deg(v)+1 : v \in \G\}) \textrm{ and} \]
    \[ \kappa = \gcd(\{|N[u] \cap N[v]| : u,v \in \G, u \neq v\}). \]
In particular, if $\G$ has an edge $u - v$ that is not part of a $3$-cycle $u - v - w - u$, e.g. if $\G$ is bipartite or has girth $4$, then $\mu \leq 2$.
\end{corollary}

\begin{proof}
The first part follows directly from \cref{thm:nbrly-mu} by taking $V = \emptyset$. Then, if $\G$ is bipartite, we have $|N[u] \cap N[v]| = 2$ for any neighboring vertices $u$ and $v$, and the second part follows.
\end{proof}

\section{Exploring products of graphs}
\label{sec:prods}

In this section, we will apply the techniques developed in \cref{sec:almost-RA} to the strong product, cartesian product, and tensor product of graphs. In many cases, we can characterize the elementary divisors of the RA matrix of the product of two graphs in terms of the elementary divisors of the RA matrix of the factors. 

\subsection{Strong Products}
\label{subsec:strong_prod}
Recall the strong product $\G_1 \boxtimes \G_2$ of two graphs $\G_1$ and $\G_2$ is the graph whose vertices are $(u_1, u_2)$ with $u_1 \in \G_1, u_2 \in \G_2$ where $(u_1, u_2)$ is adjacent to $(v_1, v_2)$ if and only if:

\begin{enumerate}[i.]
    \item $u_1 = v_1$ and $u_2$ is adjacent to $v_2$, or
    \item $u_1$ is adjacent to $v_1$ and $u_2 = v_2$, or
    \item $u_1$ is adjacent to $v_1$ and $u_2$ is adjacent to $v_2$.
\end{enumerate}

Note that in the product $\G_1 \boxtimes \G_2$, we have $N[u_1, u_2] = N[u_1] \times N[u_2]$, and given the importance of the closed neighborhoods to $G^\G$, this makes the strong product perhaps the most ``natural'' product of two graphs for our consideration. We find the following properties.

\begin{theorem}
Let $\G_1, \G_2$ be undirected simple graphs with $|\G_1| = m$ and $|\G_2| = n$.
    \begin{enumerate}[(a)]
        \item The activation matrix $A_{\G_1 \boxtimes \G_2}$ for $\G_1 \boxtimes \G_2$ is the tensor product (or Kronecker product) $A_{\G_1} \otimes A_{\G_2}$ of the activation matrices $A_{\G_1}$ for $\G_1$ and $A_{\G_2}$ for $\G_2$. 
        \item If $A_{\G_1}$ has elementary divisors $a_1 \mid a_2 \mid \cdots \mid a_m$ and $A_{\G_2}$ has $b_1 \mid b_2 \mid \cdots \mid b_n$ (where the last few $a_i$'s and/or $b_j$'s may be 0), then $A_{\G_1 \boxtimes \G_2}$ has elementary divisors $\{a_ib_j\}_{i = 1, j = 1}^{m, n}$, up to rearrangement of the prime factors.
        \item $C_{\G_1 \boxtimes \G_2} = C_{\G_1} \otimes C_{\G_2}$, so the elementary divisors of the RA matrix $C_{\G_1 \boxtimes \G_2}$ are the products of the elementary divisors of the RA matrices $C_{\G_1}$ and $C_{\G_2}$ (up to rearrangement of the prime factors).
    \end{enumerate}
\end{theorem}

\begin{proof}
    (a) If $A_{\G_1} = (a_{ij})$ and we label the vertices of $\G_1 \boxtimes \G_2$ in lexicographic order, then it is clear the activation matrix $A_{\G_1 \boxtimes \G_2}$ has the block form
    \[\begin{pmatrix}
        A_{\G_2} & a_{1,2}A_{\G_2} & \cdots & a_{1,m}A_{\G_2}\\
        a_{2,1}A_{\G_2} & A_{\G_2} & \cdots & a_{2,m}A_{\G_2}\\
        \vdots & \vdots & & \vdots\\
        a_{m,1}A_{\G_2} & a_{m,2}A_{\G_2} & \cdots & A_{\G_2}
    \end{pmatrix}.\]
    The well-known property $(A \otimes B)(C \otimes D) = AC \otimes BD$ guarantees that if $PA_{\G_1}Q = D_1$ and $RA_{\G_2}S = D_2$, with $D_1$ and $D_2$ the Smith normal forms of $A_{\G_1}$ and $A_{\G_2}$, respectively, then $(P \otimes R)A_{\G_1 \boxtimes \G_2}(Q\otimes S) = D_1 \otimes D_2$. Now, $D_1 \otimes D_2$ has precisely the entries $a_ib_j$ for $a_i$ in the diagonal of $D_1$ and $b_j$ in that of $D_2$, so up to rearrangement of the prime factors of $\{a_ib_j\}$, $D_1 \otimes D_2$ is the Smith normal form of $A_{\G_1 \boxtimes \G_2}$ and (b) follows.

    Any row of $C_{\G_1 \boxtimes \G_2}$ not in $A_{\G_1 \boxtimes \G_2}$ is the intersection of two rows of the form 
    \[(a_{u_1, 1}r_{v_1} \;\; a_{u_1, 2}r_{v_1} \; \cdots \; a_{u_1, m}r_{v_1}) \text{ and }(a_{u_2, 1}r_{v_2} \;\; a_{u_2, 2}r_{v_2} \; \cdots \; a_{u_2, m}r_{v_2})\]
    for some $u_1, u_2 \in \G_1, v_1, v_2 \in \G_2$, where $r_v$ denotes the row of $A_{\G_2}$ corresponding to vertex $v \in G_2$ (by intersection of rows, we mean the element-wise product as usual). Letting $r_{v_1} \cap r_{v_2}$ denote the intersections of the rows $r_{v_1}$ and $r_{v_2}$ and similarly $r_{u_1}$, $r_{u_2}$ denote the rows of $A_{\G_1}$ corresponding to the vertices $u_1$,$u_2$, respectively, the result is
    \[(a_{u_1,1}a_{u_2,1}(r_{v_1} \cap r_{v_2})\;\; a_{u_1,2}a_{u_2,2}(r_{v_1} \cap r_{v_2})\;\cdots \; a_{u_1,m}a_{u_2,m}(r_{v_1} \cap r_{v_2})),\]
    which is precisely the tensor product $(r_{u_1} \cap r_{u_2}) \otimes (r_{v_1} \cap r_{u_2})$. Just as in (b), (c) follows. 
\end{proof}

\subsection{Cartesian products}
\label{subsec:Cartesian}

The cartesian product $\G \bprod \Lambda$ of two graphs $\G$ and $\Lambda$ has vertex set $V(\G) \times V(\Lambda)$, and there is an edge between $(u_1, u_2)$ and $(v_1, v_2)$ whenever they agree in one coordinate and are adjacent in the other coordinate. That is, either $u_1 = v_1$ and $u_2$ and $v_2$ are neighbors, or $u_2 = v_2$ and $u_1$ and $v_1$ are neighbors. The \emph{prism over $\G$} is the graph $\G \bprod K_2$. 

We have seen that cube graphs are always either RA or $1/2$-RA. In fact, there is a much more general phenomenon at play here: the cartesian product of any bipartite graphs is either RA or $1/2$-RA. More generally, the cartesian product of any two graphs is almost RA:

\begin{theorem} \label{thm:bprod-1/mu}
The graph $\G_1 \bprod \G_2$ is neighborly. Furthermore,
\begin{enumerate}
    \item If $\G_1$ and $\G_2$ are either both bipartite or both not bipartite, then $\G_1 \bprod \G_2$ is either $1/2$-RA or RA, and it is $1/2$-RA if and only if the degrees of vertices of $\G_1$ and the degrees of vertices of $\G_2$ have opposite parity, and every pair of vertices in each $\G_i$ has an even number of common neighbors.
    \item Otherwise, if $\G_1$ is not bipartite and $\G_2$ is, then $\G_1 \bprod \G_2$ is $1/\mu$-RA, where $\mu = \gcd(\delta, \kappa_1, \kappa_2),$ with
    \[ \delta = \gcd(\{1 + \deg(u) - \deg(i) : u \in \G_1, i \in \G_2\}), \]
    \[ \kappa_{1} = \gcd(\{|N[u] \cap N[v]| : u,v \in \G_1, u \neq v\}), \textrm{ and} \]
    \[ \kappa_{2} = \gcd(\{|N[i] \cap N[j]| : i,j \in \G_2, d(i,j) = 2\}). \]
    
\end{enumerate}
\end{theorem}

\begin{proof}
Let $\G = \G_1 \bprod \G_2$. 
Consider neighbors $(u, i)$ and $(v, j)$ of $\G_1 \bprod \G_2$. Then either $u = v$ or $i=j$. Let us assume that $i = j = 1$. If $\G_1$ is bipartite, then $u$ and $v$ have no common neighbors, so $\bbe_{(u,1)} + \bbe_{(v,1)} \in \Rg$. If $\G_1$ is not bipartite, let $w$ be a neighbor of $v$ other than $u$, and let $2$ be a neighbor of $1$. (If no such neighbor $w$ exists, then again $N[u] \cap N[v] = \{u,v\}$ and thus $\bbe_{(u,1)} + \bbe_{(v,1)} \in \Rg$.) Then
\[ (\vec N[u,2] \cap \vec N[v,1]) - (\vec N[v,1] \cap \vec N[w,2]) = (\bbe_{(u,1)} + \bbe_{(v,2)}) - (\bbe_{(v,2)} + \bbe_{(w,1)}) = \bbe_{(u,1)} - \bbe_{(w,1)}.\]
Thus, whenever $u$ and $w$ are vertices connected by an even path in $\G_1$, we have $\bbe_{(u,1)} - \bbe_{(w,1)} \in \Rg$. Since $\G_1$ is not bipartite, every pair of vertices is connected by an even path, and so $\bbe_{(u,1)} - \bbe_{(w,1)} \in \Rg$ for \emph{all} vertices $u$ and $w$ of $\G_1$. Applying the same argument to $\G_2$ shows that every $\bbe_{(u,i)} + \bbe_{(u,j)} \in \Rg$ if $\G_2$ is bipartite, and otherwise every $\bbe_{(u,i)} - \bbe_{(u,j)} \in \Rg$. It follows that $\G$ is neighborly. 

By the arguments above, if $\G_1$ and $\G_2$ are both not bipartite, then $\G$ is negatively neighborly. Similarly, if $\G_1$ and $\G_2$ are both bipartite, then $\G$ is positively neighborly. In either case, since $|N[u,1] \cap N[v,2]| = 2$, \cref{cor:pos-nbr-1/mu} or \cref{thm:neg-nbr-1/mu} says that $\G$ is either $1/2$-RA or RA. Furthermore, by \cref{thm:half-ra}, $\G$ is half-RA if and only if every vertex has odd degree and every pair of vertices has an even number of common neighbors. Since $\deg(u,i) = \deg(u) + \deg(i)$, every vertex has odd degree if and only if every vertex of $\G_1$ (resp. $\G_2$) has odd degree and every vertex of $\G_2$ (resp. $\G_1$) has even degree. Now, two vertices $(u,i)$ and $(v,j)$ of $\G$ have common neighbors if and only if they agree in one coordinate or $u$ is a neighbor of $v$ and $i$ is a neighbor of $j$. We already accounted for the latter case above (such vertices have two common neighbors), so every pair of vertices will have an even number of common neighbors if and only if that is true of both $\G_1$ and $\G_2$.

Now suppose that $\G_1$ is not bipartite and $\G_2$ is. By the arguments above, if $i$ and $j$ are connected by a path of length $k$ in $\G_2$, then $\bbe_{(u,i)} + (-1)^{k-1} \bbe_{(v,j)} \in \Rg$ for every pair of vertices $u$ and $v$ in $\G_1$. The bipartition of $\G_2$ induces a partition $U \sqcup V$ of $\G_1 \bprod \G_2$ by simply ignoring the first coordinate of each vertex, and this partition satisfies the conditions of \cref{thm:nbrly-mu}. Each vertex $(u,i)$ has $\deg(u) + \deg(i)$ neighbors, with $\deg(u)$ of them in the same part (with the same second coordinate) and $\deg(i)$ of them in the opposite part (with the same first coordinate). That shows that $\delta$ as defined here is the same as the $\delta$ in \cref{thm:nbrly-mu}. 

It remains to compute $\kappa$ from \cref{thm:nbrly-mu}. 
Consider vertices $(u,i) \neq (v,j)$ with common neighbors. If they do not agree in either coordinate, then $u$ and $v$ are neighbors, $i$ and $j$ are neighbors, and $N[u,i] \cap N[v,j] = \{(u,j), (v,i)\}$. The two vertices in this intersection lie in different parts, and so they contribute a 0 to the set whose gcd is $\kappa$, leaving it unaffected. Thus, it suffices to consider vertices that agree in one coordinate. The vertices of $N[u,i] \cap N[v,i]$ all have second coordinate $i$ and thus all lie in the same part, so the contribution of this set to $\kappa$ is just $|N[u,i] \cap N[v,i]| = |N[u] \cap N[v]|$. Now, suppose that $(u,i)$ and $(u,j)$ have common neighbors. If they are adjacent, then since $\G_2$ is bipartite (with bipartition induced by $U \sqcup V$), they are the only vertices of $N[u,i] \cap N[u,j]$, which thus has one vertex in $U$ and one in $V$, contributing nothing to $\kappa$. Thus we may assume that $(u,i)$ and $(u,j)$ are non-adjacent, with all their common neighbors lying in the same part as each other. So again, the contribution of this set to $\kappa$ is just $|N[u,i] \cap N[u,j]| = |N[i] \cap N[j]|$. Thus, $\kappa = \gcd(\kappa_1, \kappa_2)$, and the result follows.
\end{proof}

\begin{corollary} \label{cor:prism-1/mu}
If $\G$ is not bipartite, then the prism $\G \bprod K_2$ is $1/\mu$-RA, where $\mu = \gcd(\delta, \kappa)$ with
\[ \delta = \gcd(\{\deg(v) : v \in \G\}), \qquad \kappa = \gcd(\{|N[u] \cap N[v]| : u,v \in \G\}). \]
\end{corollary}

\begin{corollary} \label{cor:girth-5}
If $\G_1$ has non-adjacent vertices $u$ and $v$ that have a unique common neighbor, then $\G_1 \bprod \G_2$ is RA. In particular, if $\G_1$ has girth 5 or more and $|\G_1| > 2$, then $\G_1 \bprod \G_2$ is RA.
\end{corollary}

\begin{proof}
If $\G_1$ and $\G_2$ are both bipartite or both not bipartite, then this follows directly from the first case of \cref{thm:bprod-1/mu}. Otherwise, in the second case, we see that $\kappa_1 = 1$ or, if $\G_1$ is bipartite and $\G_2$ is not, then we relabel them and $\kappa_2 = 1$. In each case, $\G_1 \bprod \G_2$ is RA.
\end{proof}

\begin{corollary} \label{cor:girth-4}
If $\G_1$ and $\G_2$ are both girth $3$ or both girth $4$, then $\G_1 \bprod \G_2$ is either $1/2$-RA or RA.
\end{corollary}

\begin{proof}
This follows immediately from the statement of \cref{thm:bprod-1/mu} except when $\G_1$ is not bipartite and $\G_2$ is (and both are girth $4$). Then $\kappa_1 \leq 2$ since for adjacent vertices $u$ and $v$ in a girth 4 graph, $|N[u] \cap N[v]| = 2$.
\end{proof}

Together \cref{cor:girth-5,cor:girth-4} imply that the only way to get $\mu > 2$ in \cref{thm:bprod-1/mu} is if either $\G_2 = K_2$ or if $\G_1$ has girth 3 and $\G_2$ has girth 4.

Recall that for every group $G$, the group $G^{K_n}$ is the diagonal subgroup of $G^n$, isomorphic to just $G$ itself. In particular, $K_n$ is as far from being RA as possible -- $C_{K_n}$ has only a single nontrivial elementary divisor. It is perhaps surprising then that the cartesian product of complete graphs is always $1/2$-RA or RA:

\begin{corollary} \label{kn-prods-ra}
    Let $K = K_{n_1} \bprod K_{n_2} \bprod \cdots \bprod K_{n_m}$, with $m \ge 2$ and each $n_i \geq 2$. Then $K$ is $1/2$-RA or RA, and it is RA if and only if $m$ is even or at least one $n_i$ is odd. 
\end{corollary}

\begin{proof}
Up to rearrangement, we may assume that $n_1 \geq n_2 \geq \cdots \geq n_m$.

First, suppose that each $n_i = 2$. Then $K$ is the cube graph $Q_m$, which is RA if and only if $m$ is even.

Next, suppose that $n_1 > 2$ and $n_2 = n_3 = \cdots = n_m = 2$. Then $K_{n_1}$ is not bipartite and $(K_{n_2} \bprod \cdots \bprod K_{n_m})$ is. Applying \cref{thm:bprod-1/mu} gives us $\delta = n_1 - m + 1$, $\kappa_1 = n_1$, and $\kappa_2 = 2$. Thus $\mu = \gcd(\delta, \kappa_1, \kappa_2)$ will be $2$ if and only if $n_1$ is even and $m$ is odd, and otherwise $\mu = 1$.

Finally, suppose that $n_1 > 2$ and $n_2 > 2$. Then $K_{n_1}$ and  $(K_{n_2} \bprod \cdots \bprod K_{n_m})$ are both not bipartite, so $K$ is $1/2$-RA or RA. Let $K' = K_{n_2} \bprod \cdots \bprod K_{n_m}$. The degree of each vertex of $K_{n_1}$ is $n_1 - 1$, and the degree of each vertex of $K'$ is $\sum_{i=2}^m n_i - (m-1)$. Every pair of vertices of $K_{n_1}$ has $n_1$ common neighbors (including each other), and some pairs of vertices of $K'$ have $2$ common neighbors, with others having $n_i$ common neighbors for some $2 \leq i \leq m$. Thus, $K$ will be $1/2$-RA if and only if every $n_i$ is even and $m$ is odd.
\end{proof}

\begin{example}
    \label{eg:CylTori}
    Let us consider the higher-dimensional analogue of grids / cylinders / tori
    \[\G = P_{m_1} \bprod \cdots \bprod P_{m_k} \bprod C_{n_1} \bprod \cdots \bprod C_{n_t},\] where $k$ or $t$ could be 0 and all $n_j$ are equal to 3 or at least 5 ($C_2 = P_2$ and $C_4 = P_2 \bprod P_2$ are lumped in with the path graphs). Paths with at least 3 vertices and cycles $C_n$ with $n \ge 5$ have girth at least 5, so if some $m_i \ge 3$ or some $n_j \ge 5$ then $\G$ is RA by \cref{cor:girth-5}. Otherwise, each $m_i = 2$ and each $n_j = 3$ so we have a cartesian product of complete graphs and \cref{kn-prods-ra} tells us $\G$ is RA if and only if $t \ge 1$ or $t = 0$ (so $\G = Q_k$) and $k$ is even. Thus, the only case when such a graph $\G$ is not RA is when $\G = Q_k$ with $k$ odd, and then $\G$ is $1/2$-RA.
\end{example}

Let us now show one fun example of an infinite family of non-RA graphs built from Kneser graphs.

\begin{example} \label{Kneser-not-RA}
Consider the prism $\Kn(n,k) \bprod K_2$. Let us show that, for infinitely many choices of $n$ and $k$, the result is $1/3$-RA. To do so, we will use \cref{cor:prism-1/mu} and prove that $\mu = 3$. 

The degree of any vertex of the Kneser graph $\Kn(n, k)$ is ${n-k \choose k}$. Any two non-adjacent vertices correspond to subsets of $\{1, 2, \dots, n\}$ that share $j$ elements, with $1 \le j \le k-1$, and the number of vertices adjacent to such two vertices is ${n-2k+j \choose k}$. Any two adjacent vertices correspond to disjoint subsets of $\{1, 2, \dots, n\}$ and have ${n - 2k \choose k}$ vertices in common (other than themselves). Thus, we are looking for pairs $(n, k)$ such that 
\begin{enumerate}[(a)]
    \item For any $j = 1, \dots, k$ we have ${n - 2k + j \choose k} \equiv 0 \pmod 3$ and
    \item ${n - 2k \choose k} \equiv 1 \pmod 3$.
\end{enumerate}

The key point will be the self-similarity of Pascal's Triangle $\pmod p$ for a prime $p$ (see \cref{Pascal10rows-mod3} and, e.g. \cite{Granville2019} and \cite{Kubelka2004}): take $\Delta$ = the first $p$ rows of Pascal's triangle and replace each entry by that entry multiplied by $\Delta$, putting in upside-down triangles of 0's in between the resulting triangles to get the first $p^2$ rows, then do the same, either multiplying the entries of the first $p$ rows by the entire $p^2$ first rows or multiplying the entries of the first $p^2$ rows by $\Delta$, to find the first $p^3$ rows, and so on. See \cref{Pascal10rows-mod3}.

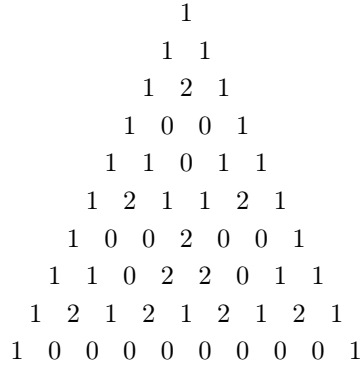
\begin{figure}[hbtp]
    \centering
\begin{tikzpicture}[scale=0.5]
    \node at (0, 0) {1};
    
    \node at (-0.5, -1) {1};
    \node at (0.5, -1) {1};
    
    \node at (-1, -2) {1};
    \node at (0, -2) {2};
    \node at (1, -2) {1};
    
    \node at (-1.5, -3) {1};
    \node at (-0.5, -3) {0};
    \node at (0.5, -3) {0};
    \node at (1.5, -3) {1};
    
    \node at (-2, -4) {1};
    \node at (-1, -4) {1};
    \node at (0, -4) {0};
    \node at (1, -4) {1};
    \node at (2, -4) {1};
    
    \node at (-2.5, -5) {1};
    \node at (-1.5, -5) {2};
    \node at (-0.5, -5) {1};
    \node at (0.5, -5) {1};
    \node at (1.5, -5) {2};
    \node at (2.5, -5) {1};
    
    \node at (-3, -6) {1};
    \node at (-2, -6) {0};
    \node at (-1, -6) {0};
    \node at (0, -6) {2};
    \node at (1, -6) {0};
    \node at (2, -6) {0};
    \node at (3, -6) {1};
    
    \node at (-3.5, -7) {1};
    \node at (-2.5, -7) {1};
    \node at (-1.5, -7) {0};
    \node at (-0.5, -7) {2};
    \node at (0.5, -7) {2};
    \node at (1.5, -7) {0};
    \node at (2.5, -7) {1};
    \node at (3.5, -7) {1};
    
    \node at (-4, -8) {1};
    \node at (-3, -8) {2};
    \node at (-2, -8) {1};
    \node at (-1, -8) {2};
    \node at (0, -8) {1};
    \node at (1, -8) {2};
    \node at (2, -8) {1};
    \node at (3, -8) {2};
    \node at (4, -8) {1};
    
    \node at (-4.5, -9) {1};
    \node at (-3.5, -9) {0};
    \node at (-2.5, -9) {0};
    \node at (-1.5, -9) {0};
    \node at (-0.5, -9) {0};
    \node at (0.5, -9) {0};
    \node at (1.5, -9) {0};
    \node at (2.5, -9) {0};
    \node at (3.5, -9) {0};
    \node at (4.5, -9) {1};
    
\end{tikzpicture}
    \caption{Rows 0 - 9 of Pascal's Triangle $\pmod 3$}
    \label{Pascal10rows-mod3}
\end{figure}    

    Note that for any $a \ge 0$, the $3^{a+1}$-th row ($1,0,\dots,0,1$) marks the start of an upside down triangle of zeroes of size $3^{a+1} - 1$ (in width and height), so the zeroes extend from the ${3^{a+1} \choose r}$ entry down-left to the ${3^{a+1} + r-1 \choose r}$ entry (inclusive) for $r = 1, 2, \dots, 3^{a+1} - 1$. Then if $n = 3^{a+1} + 2k - 1$ and $k < 3^{a+1}$, we get
    \[ {n-2k+j \choose k} = {3^{a+1} + j - 1 \choose k} \equiv 0 \pmod 3\]
    for each $j$ with $1 \leq j \leq k$. This proves (a) above. Then note that
    \[ {n - 2k \choose k} = {3^{a+1} - 1 \choose k},\]
    which we want to be $1$. The $(3^a - 1)$st row follows the alternating pattern $1, 2, 1, 2, \dots, 2, 1$, meaning the entry before the zeroes start for even $k$ is 1. So, picking $k$ to be an even number between $3^a$ and $3^{a+1}$ satisfies (b). In other words, picking any nonnegative integers $a$ and $b$, we can set $k = 3^a + 1 + 2b$ and $n = 3^{a+1} + 2k - 1$, and then $\Kn(n,k)$ has properties (a) and (b). In fact, the numbers in column $k$ repeat with period $3^{a+1}$ so the statements above still hold if we add $3^{a+1}i$ to $n$ for any nonnegative integer $i$.

    \begin{example}
        Prisms over complements of Kneser graphs serve as an example of non-RA graphs as well. Let $n \ge 3$, $K$ be the complement of $\Kn(n, 2)$ and $\G = K \bprod K_2$. Then $\G$ is 1/4-RA if $n \equiv 0 \pmod 4$, 1/2-RA if $n \equiv 2 \pmod 4$, and RA if $n$ is odd. Indeed,
        \begin{enumerate}[(a)]
            \item The degree of a vertex, say $\{1, 2\}$, is $2(n-2)$ since it has adjacent vertices $\{1, j\}$ and $\{2, j\}$ for $3 \le j \le n$.
            \item The intersection $N(u) \cap N(v)$ for two adjacent vertices, say $\{1, 2\}$ and $\{1, 3\}$, has $n$ vertices: $u$ and $v$ themselves, $\{2, 3\}$, and $\{1, j\}$ for $4 \le j \le n$.
            \item The intersection $N(u) \cap N(v)$ for two non-adjacent vertices, say $\{1, 2\}$ and $\{3, 4\}$, has exactly 4 vertices: $\{1, 3\}, \{1, 4\}, \{2, 3\}, \{2, 4\}$.
        \end{enumerate}
        Taking the gcd of these numbers and using \cref{cor:prism-1/mu} completes the proof.
    \end{example}


    \end{example}

\subsection{Tensor Products}
\label{subsec:tensor_prod}

The tensor product $\G_1 \times \G_2$ of two graphs $\G_1$ and $\G_2$ is the graph whose vertices again are $(u_1, u_2)$ with $u_1 \in \G_1, u_2 \in \G_2$, and where $(u_1, u_2)$ is adjacent to $(v_1, v_2)$ if and only if $u_1$ is adjacent to $v_1$ \textit{and} $u_2$ is adjacent to $v_2$. It is well-known that the tensor product of two connected graphs is disconnected if and only if they are both bipartite; first, we will deal with the case that one of the graphs is bipartite, then the case where both are, and last with the case that neither are.

\begin{theorem}
\label{thm:tensor-bipartite}
    Let $\G$ be a non-bipartite graph and $\Lambda = \Lambda_1 \sqcup \Lambda_2$ a bipartite graph. Then $\G \times \Lambda$ is $1/\mu$-RA with $\mu = \gcd(\delta, \kappa)$, where
    \[\delta = \gcd\{\deg(v)\deg(\lambda)-1 : v \in \G, \lambda \in \Lambda\} \textrm{ and}\]
    \[\kappa = \gcd\{|N(u) \cap N(v)| \cdot |N(\lambda_1) \cap N(\lambda_2)|: u, v \in \G, \lambda_1, \lambda_2 \in \Lambda_i, i = 1, 2, (u, \lambda_1) \ne (v, \lambda_2)\}.\] 
    In particular, $\G \times K_2$ is $1/\mu$-RA, where $\mu = \gcd(\delta, \kappa)$ with
    \[\delta = \gcd\{\deg(v) - 1 : v \in \G\}, \qquad \kappa = \gcd\{|N(u) \cap N(v)| : u, v \in \G\}.\]
\end{theorem}
\begin{proof}
    Let $1$ and $2$ denote adjacent vertices in $\Lambda$. If $u$ and $v$ are neighbors in $\G$, then $N[u,1] \cap N[v,2]$ = $\{(u,1), (v,2)\}$ since $1$ and $2$ have no common neighbors. Then if we take a path $u-v-w$ in $\G$, and take $(\vec N[u,1] \cap \vec N[v,2]) - (\vec N[v,2] \cap \vec N[w,1])$, we see that $\bbe_{(u,1)} - \bbe_{(w,1)} \in \Z^{C_{\G \times \Lambda}}$. By similar arguments, if $u$ and $w$ are any two vertices of $\G$ that are an even distance apart, we have $\bbe_{(u,1)} - \bbe_{(w,1)} \in \Z^{C_{\G \times \Lambda}}$. Now $\G$ is not bipartite so every pair of vertices of $\G$ is connected by an even-length path. Thus, for any $u, v \in \G$ and $i = 1, 2$ we have $\bbe_{(u, i)} - \bbe_{(v, i)} \in \Z^{C_{\G \times \Lambda}}$. Coming back to neighbors $u$ and $v$, this, together with the fact that $\bbe_{(u,1)} + \bbe_{(v,2)} \in \Z^{C_{\G \times \Lambda}}$, tells us $\bbe_{(u,1)} + \bbe_{(u,2)} \in \Z^{C_{\G \times \Lambda}}$. Continuing along the edges of $\Lambda$, we see that $\bbe_{(u,\lambda_1)} - \bbe_{(u,\lambda_2)} \in \Z^{C_{\G \times \Lambda}}$ if $\lambda_1, \lambda_2 \in \Lambda$ are an even distance apart (i.e. in the same part of $\Lambda$), whereas $\bbe_{(u,\lambda_1)} + \bbe_{(u,\lambda_2)} \in \Z^{C_{\G \times \Lambda}}$ if $\lambda_1$ and $\lambda_2$ are  an odd distance apart (i.e. in different parts of $\Lambda$). Thus, taking $U = \{(v, \lambda) : v \in \G, \lambda \in \Lambda_1\}$ and $V = \{(v, \lambda) : v \in \G, \lambda \in \Lambda_2\}$, we can apply \cref{thm:nbrly-mu} to see that $\Gamma \times \Lambda$ is almost RA, as follows.
    
    For every $v \in \G,\lambda \in \Lambda$, $N[v,\lambda]$ has 1 vertex in one part and $\deg(v)\deg(\lambda)$ vertices in the other, so $\delta = \gcd\{\deg(v)\deg(\lambda)-1 : v \in \G, \lambda \in \Lambda\} = 0$ only if $\deg(\lambda) = \deg(v) = 1$ for all $v \in \G$ and $\lambda \in \Lambda$. But this means $\G = K_2$ is bipartite --- a contradiction, so $\delta \ne 0$ and $\Gamma \times \Lambda$ is $1/\mu$-RA, where $\mu = \gcd(\delta, \kappa)$. To write down $\kappa$, note that if $\lambda_1$ and $\lambda_2$ are in different parts of $\Lambda$, then $N[u,\lambda_1] \cap N[v,\lambda_2]$ is $\emp$ if $u$ and $v$ are not adjacent, or else contains $(u, \lambda_1)$ in one part and $(v, \lambda_2)$ in the other, in which case the difference in the sizes of intersections of $N[u,\lambda_1] \cap N[v,\lambda_2]$ with $U$ and $V$ is $1 - 1 = 0$. Thus, the only parts that contribute to $\kappa$ are where $\lambda_1$ and $\lambda_2$ are in the same part of $\Lambda$, and in that case all of $N[u,\lambda_1] \cap N[v,\lambda_2]$ is contained in one part, so the difference we seek is $|N(u) \cap N(v)| \cdot |N(\lambda_1) \cap N(\lambda_2)|$. Noting that $|N(\lambda_1) \cap N(\lambda_2)| = 1$ and $\deg(\lambda) = 1$ when $\Lambda = K_2$, this proves the theorem. 
\end{proof}

Taking the tensor product of a Kneser graph with $K_2$ often gives an example where \cref{thm:tensor-bipartite} gives a nontrivial $\Rg$. The following general result follows from our work and \cite{GCD}.

\begin{corollary}    
    Let $\G = \Kn(n, k) \times K_2$, with $n > 2k$. Let $L = \lcm\{1, 2, \dots, k\}$ and $n' = n - 2k$. Then $\G$ is $1/\mu$-RA, where $\mu = n'/\gcd(L, n')$. 
\end{corollary}
\begin{proof}
    Recall from \cref{Kneser-not-RA} that the degree of every vertex $v$ is ${n - k \choose k}$ and the number of common neighbors of two vertices that share $i$ numbers in their label, say $\{1, \dots, k\}$ and $\{1, \dots, i, k+1, \dots, 2k-i\}$, is given by ${n - 2k + i \choose k}$, where $i \in \{0, 1, \dots, k-1\}$. Thus, by \cref{thm:tensor-bipartite},
    \[\mu = \gcd\left\{{n-2k \choose k}, \dots, {n-k-1 \choose k}, {n-k \choose k} - 1\right\}.\]
    
    

    Repeatedly applying the fact that $\gcd(b, a+b) = \gcd(a, b)$ to the binomial coefficients in the manner of
    \[\gcd\left({m \choose k}, {m+1 \choose k}\right) = \gcd\left({m \choose k}, {m \choose k-1} + {m \choose k}\right) = \gcd\left({m \choose k}, {m \choose k-1}\right)\]
    and 
    \[\gcd\left({m \choose k}, {m+1 \choose k} - 1\right) = \gcd\left({m \choose k}, {m \choose k-1} + {m \choose k} - 1\right) = \gcd\left({m \choose k}, {m \choose k-1} - 1\right)\]
    shows that the desired is actually the gcd of
    \[{n-2k \choose k}, \dots, {n-2k \choose 2}, {n-2k \choose 1}, {n-2k \choose 0} - 1 = 0.\]
    The result follows from \cite[Theorem 1]{GCD}, which  states the gcd of the nonzero numbers above equals $d(n-2k; 1, k) = (n-2k)/\gcd(L, n-2k)$. 
    \end{proof}


Now let us consider the case where both $\Gamma = \Gamma_1 \sqcup \Gamma_2$ and $\Lambda = \Lambda_1 \sqcup \Lambda_2$ are bipartite. 
In this case, their tensor product has two connected components: the first is the bipartite graph with parts $\Gamma_1 \times \Lambda_1$ and $\Gamma_2 \times \Lambda_2$, and the second is bipartite with parts $\Gamma_1 \times \Lambda_2$ and $\Gamma_2 \times \Lambda_1$. 
We will work with the first connected component and keep in mind that the same arguments can be applied to the second. 
Just as in the proof of \cref{thm:tensor-bipartite}, if $u, v \in \G$ are neighbors and $\la_1, \la_2 \in \Lambda$ are neighbors, then $N[u, 1] \cap N[v, 2] = \{(u, 1), (v, 2)\}$ and if we have paths $u - v - w$ in $\G$ and $\la_1 - \la_2 - \la_3$ in $\La$, then $(\vec N[u, \la_1] \cap \vec N[v, \la_2]) - (\vec N[v, \la_2] \cap \vec N[w, \la_3]) = \vec e_{(u, \la_1)} - \vec e_{(w, \la_3)}\}$. But for either $i \in \{1, 2\}$ and $j \in \{1, 2\}$, if $u, w \in \Gamma_i$ and $\la_1, \la_3 \in \Lambda_j$ then there is an even-length path between $(u, \la_1)$ and $(w, \la_3)$ (take any even length path between $u$ and $w$ in $\G$ and follow it with second coordinates swapping between $\la_1$ and its neighbor in $\Lambda$, and then do the same for going from $\la_1$ to $\la_3$, swapping between $w$ and its neighbor in $\G$ for the first coordinate). Similarly, there is an odd-length path between $(u, \la_1)$ and $(v, \la_2)$ if $u$ and $v$ are in opposite parts of $\G$ and $\la_1, \la_2$ are in opposite parts of $\La$. So again, $\vec e_{(u, \la_1)} + \vec e_{(v, \la_2)} \in \Z^{C_{\G \times \Lambda}}$ if $(u, \la_1)$ and $(v, \la_2)$ are in opposite parts of a connected component of $\G \times \Lambda$, and $\vec e_{(u, \la_1)} - \vec e_{(w, \la_3)} \in \Z^{C_{\G \times \Lambda}}$ if $(u, \la_1)$ and $(w, \la_3)$ are in the same part of a connected component. That means we can apply \cref{thm:nbrly-mu} to each of the connected components of $\G \times \Lambda$. By the same considerations as in the proof of \cref{thm:tensor-bipartite}, this gives us the following.

\begin{theorem}
\label{thm:tensor-2bipartite}
    Let $\G = \Gamma_1 \sqcup \Gamma_2$ and $\Lambda = \Lambda_1 \sqcup \Lambda_2$ be bipartite graphs. Then $\G \times \Lambda$ is the disjoint union of two connected components, one $1/\mu_1$-RA and the other $1/\mu_2$-RA, with $\mu_i = \gcd(\delta_i, \kappa_i)$, where
    \[\delta_i = \gcd\{\deg(v)\deg(\lambda)-1 : v \in \G_a, \lambda \in \Lambda_b, a \in \{1, 2\}\},\]
    \[\kappa_i = \gcd\{|N(u) \cap N(v)| \cdot |N(\lambda_1) \cap N(\lambda_2)|: u, v \in \G_a, \lambda_1, \lambda_2 \in \Lambda_b, a \in \{1, 2\}, (u, \lambda_1) \ne (v, \lambda_2)\},\] 
    and $b = a$ (resp. $b = 3-a$) if $i = 1$ (resp. $i = 2$). 
\end{theorem}

For example, $C_\G$ for $\G = K_2 \times \Cr(2k+4)$ has two elementary divisors equal to $k$ for any $k \ge 1$. This makes sense because, as we will see in \cref{cor:crown-graph}, the crown graph $\Cr(2k+4)\equiv K_2 \times K_{k+2}$ is $1/k$-RA and $\G$ has two disconnected components, each isomorphic to $\Cr(2k+4)$. In general, for $\G = K_2 \times \Lambda$ where $\Lambda$ is a bipartite graph, $\G$ is isomorphic to two disjoint copies of $\Lambda$ and $C_\G$ either has all 1's as elementary divisors (if $\Lambda$ is RA) or has two elementary divisors equal to $\mu$ if $\Lambda$ is $1/\mu$-RA.

Finally, we explore the tensor product of two non-bipartite graphs.


\begin{theorem}
\label{thm:K_mxnot-bipartite}
If $\G$ is not bipartite and $m \geq 3$, then $\G \times K_m$ is $1/2$-RA or RA, and it is $1/2$-RA if and only if $m$ is even, $\deg v$ is odd for every $v \in \G$, and $|N[u] \cap N[v]|$ is even for every $u, v \in \G$.
\end{theorem}

\begin{proof}
First, note that for any $u \in \G$,
\begin{equation}
    \label{eqn:tensor:vertical-move}
    \vec N[u,1] - \vec N[u,2] + (\vec N[u,2] \cap \vec N[u,3]) - (\vec N[u,1] \cap \vec N[u,3]) = \bbe_{(u,1)} - \bbe_{(u,2)} \in \Z^{C_{\G \times K_m}},
\end{equation}
which means we can move the coefficient of any vertex in a row of $C_{\G \times K_m}$ to a vertex with any other second coordinate. 

Next, note that 
\begin{equation}
    \label{eqn:tensor:mu|m-2}
    (m-2) \vec N[u,1] - (\vec N[u,1] \cap \vec N[u,2]) - \ldots - (\vec N[u,1] \cap \vec N[u,m]) = (m-2)\bbe_{(u,1)} \in \Z^{C_{\G \times K_m}}.
\end{equation}
Furthermore, for any neighbors $u, v \in \G$, we have
\begin{equation}
    \label{eqn:tensor:(m-2)intersections}
    (m-2)(\vec N[u,1] \cap \vec N[v,1]) - (\vec N[u,1] \cap \vec N[v,2]) - \ldots - (\vec N[u,1] \cap \vec N[v,m]) = (1-m)\bbe_{(u,1)} - \sum_{i=2}^m \bbe_{(v,i)} \in \Z^{C_{\G \times K_m}}.
\end{equation}

From \eqref{eqn:tensor:vertical-move}, moving all the $(v, i)$ coefficients to that of $(v, 1)$ by adding $\sum_{i=2}^m (\bbe_{(v,i)} - \bbe_{(v,1)})$, we get 
\[(1-m) (\bbe_{(u,1)} + \bbe_{(v,1)}) \in \Z^{C_{\G \times K_m}}.\]
Adding this to $(m-2)\bbe_{(u, 1)} + (m-2)\bbe_{(v, 1)}$ from \eqref{eqn:tensor:mu|m-2}, we find that $\bbe_{(u, 1)} + \bbe_{(v, 1)} \in \Z^{C_{\G \times K_m}}$. This means the edges along the first coordinate are ``positive'', so overall $\G \times K_m$ is neighborly. But $\G$ is not bipartite, so going around an odd cycle also gives us $\bbe_{(u,1)} - \bbe_{(v,1)} \in \Z^{C_{\G \times K_m}}$. That means, not only is $\G \times K_m$ negatively neighborly but also, combined with what we already had gives us $2\bbe_{(u,1)} \in \Z^{C_{\G \times K_m}}$, so $\G \times K_m$ is $1/2$-RA or RA.

Applying \cref{thm:neg-nbr-1/mu}, we first look at $\delta = \gcd\{\deg(v, i) + 1: v \in \G, i \in K_m\}$. We have $\deg(v, i) = (\deg v)(m-1) + 1$, which is even iff $m$ is even and $\deg v$ is odd. Next, considering $\kappa = \gcd\{|N[u,i] \cap N[v,j]| : u,v \in \G, i, j \in K_m, (u, i) \ne (v, j)\}$, if $i \ne j$ then $|N[u,i] \cap N[v,j]| = |N[u] \cap N[v]|(m-2) + 2b$, where $b = 1$ if $u$ and $v$ are neighbors or else $b = 0$. If instead $i = j$, then we have $|N[u,i] \cap N[v,i]| = |N[u] \cap N[v]|(m-1)$, so we see that $|N[u] \cap N[v]|$ is multiplied by both $m-1$ and $m-2$ (individually) in this set, which means $\kappa$ is even iff $|N[u] \cap N[v]|$ is always even. Thus, $\mu = \gcd(\delta, \kappa)$ is even iff all these conditions hold.
\end{proof}

Note in particular that \cref{thm:K_mxnot-bipartite} states that if $\G$ is not bipartite, then $\G \times K_m$ is always RA for odd $m$.

Now we turn our attention to the general tensor product of an arbitrary number of complete graphs. Note that $K_2 \times K_2$ is disconnected, and in general the tensor product $K_2 \times K_2 \times \cdots \times K_2$ of $n$ copies of $K_2$ has $2^{n-1}$ disconnected components; these components are just the pairs of opposite vertices of an $n$-dimensional cube. The tensor product of $n$ copies of $K_2$ and a non-bipartite graph $\G$ results in $2^{n-1}$ isomorphic disjoint copies of $\G$, so it suffices to consider tensor products $K_{m_1} \times K_{m_2} \times \cdots \times K_{m_n}$, with $m_1 \le m_2 \le \cdots \le m_n$, where either $m_1 > 2$ or $m_1 = 2$ and $m_2 > 2$.

\begin{theorem}
    \label{thm:tensor-completes}
    The graph $\G = K_{m_1} \times K_{m_2} \times \cdots \times K_{m_n}$, for $2 \le m_1 \le m_2 \le \cdots \le m_n$ and $m_2 \ge 3$ is always $1/\mu$-RA, where
    \begin{enumerate}[(a)]
        \item if $m_1 = 2$, then $\mu = \gcd\{m_i-2 \mid i = 2,\dots, n\}$;
        \item otherwise, $\mu \le 2$ with $\mu = 2$ if and only if all $m_i$ are even.
    \end{enumerate}
\end{theorem}
\begin{proof}
    For the case (a) $m_1 = 2$, we could write down a bit of a messy proof  using \cref{thm:tensor-bipartite}, but the cleanest way is to simply use \cref{cor:xK_{k+2}} below, so we defer until then (see \cref{rmk:cor-proves-thm}). Now, for (b), let $m_1 \ge 3$. By \cref{thm:K_mxnot-bipartite}, $\G$ is 1/2-RA or RA, and if at least one $m_i$ is odd, the graph is RA. If all the $m_i$ are even, the degree of a vertex is $\prod_{i=1}^n(m_i-1)$, which is odd, and the intersection of the closed neighborhoods of two distinct vertices has size $\prod_{i=1}^n(m_i - a_i) + 2b$, where $a_i = 1$ if the vertices agree in coordinate $i$ or else $a_i = 2$, and $b = 0$ unless the vertices differ in all coordinates, in which case $b = 1$ (to account for the two vertices themselves being in the intersection). The vertices must differ in at least one coordinate, so at least one $a_i$ must be even and hence $\prod_{i=1}^n(m_i - a_i)$ is even so $\G$ is 1/2-RA. \qedhere




\end{proof}

We have been unable to show that an arbitrary tensor product of two non-bipartite graphs is almost RA, but the following leads to a partial result in this direction (\cref{prop:non-bip-tensor}).

\begin{lemma}
    \label{lem:edge-wo-tri}
    If $\La$ is a graph with an edge $\la_1 - \la_2$ that is not part of a 3-cycle, then for any non-bipartite graph $\G$, any two vertices $u, v \in \G$ and any $\la \in \La$ we have $\vec e_{(u, \la)} - \vec e_{(v, \la)} \in \Z^{C_{\G \times \La}}$.
\end{lemma}

\begin{proof}
    If the edge $\la_1 - \la_2$ is not part of a 3-cycle, then for any edge $u - v \in \G$, $(u, \la_1)$ and $(v, \la_2)$ have no common neighbors, so $\vec N[u, \la_1] \cap \vec N[v, \la_2] = \vec e_{(u, \la_1)} + \vec e_{(v, \la_2)} \in \Z^{C_{\G \times \La}}$. For any other edge $v - w$ in $\G$, then, we have 
    \[(\vec e_{(u, \la_1)} + \vec e_{(v, \la_2)}) - (\vec e_{(v, \la_2)} + \vec e_{(w, \la_1)}) = \vec e_{(u, \la_1)} - \vec e_{(w, \la_1)} \in \Z^{C_{\G \times \La}}.\]
    Now $\G$ is non-bipartite, meaning there is an odd cycle we can go around so any two vertices are connected by an even-length path and for any two vertices $u, v \in \G$ we have $\vec e_{(u, \la_1)} - \vec e_{(v, \la_1)} \in \Z^{C_{\G \times \La}}$.

    Furthermore, if $u - v$ is an edge in $\G$ and $\la_1 - \la_2$ an edge in $\La$, we have the vector 
    \[\x(u, v, \la_1, \la_2) := (\vec e_{(u, \la_1)} - \vec e_{(v, \la_1)}) - (\vec e_{(u, \la_2)} - \vec e_{(v, \la_2)}) = \vec N[u, \la_1] \cap \vec N[v, \la_2] - \vec N[u, \la_2] \cap \vec N[v, \la_1] \in \Z^{C_{\G \times \La}}.\]
    Using the identity
    \[\x(u, v, \la_1, \la_2) + \x(u, v, \la_2, \la_3) = \x(u, v, \la_1, \la_3)\]
    along any path from $\la_1$ to $\la \in \La$, we find $\x(u, v, \la_1, \la) \in \Z^{C_{\G \times \La}}$, and subtracting this from $\vec e_{(u, \la_1)} - \vec e_{(v, \la_1)}$ gives $\vec e_{(u, \la)} - \vec e_{(v, \la)} \in \Z^{C_{\G \times \La}}$.
\end{proof}

\begin{proposition}
    \label{prop:non-bip-tensor}    
    If $\G$ and $\La$ are non-bipartite graphs that each have an edge that is not part of a 3-cycle, then $\G \times \La$ is 1/2-RA or RA.
\end{proposition}
\begin{proof}
    Apply \cref{lem:edge-wo-tri} to both graphs shows that $\G \times \La$ is negatively neighborly and if $u - v$ is an edge in $\G$ that is not part of a triangle, then for any edge $\la_1 - \la_2$ in $\La$ we have $\vec N[u, \la_1] \cap \vec N[v, \la_2] = \vec e_{(u, \la_1)} + \vec e_{(v, \la_2)} \in \Z^{C_{\G \times \La}}$ and again applying \cref{lem:edge-wo-tri} twice to this expression, we have $2\vec e_{(u, \la_1)} \in \Z^{C_{\G \times \La}}$ so $\G \times \La$ is 1/2-RA or RA (by \cref{thm:half-ra}).
\end{proof}

In conclusion, we have shown that various tensor products are $1/\mu$-RA and described $\mu$: the tensor product of a bipartite graph with a non-bipartite graph in \cref{thm:tensor-bipartite}, that of two bipartite graphs in \cref{thm:tensor-2bipartite}, an arbitrary tensor product of complete graphs in \cref{thm:tensor-completes}, and more. We have not been able to deduce such a general result for arbitrary tensor products of two non-bipartite graphs, though we will return to this question again briefly in \cref{rmk:nonbipartite-girth4-tensor}, where we show that the result is 1/2-RA or RA if at least one of the graphs has girth 4 or more. We leave the general question an open problem (\cref{prob:nonbip-tensor}).


\section{Girth $4^+$ graphs are almost RA}\label{sec:girth-4}

    In \cite[Cor. 5.13]{FirstPaper}, we showed that every graph of girth 5 or higher is RA. On the other hand, \cite[Cor. 6.14]{FirstPaper} showed that the cube graph $Q_n$ is not RA when $n$ is odd. In this section, we will further explore graphs of girth 4.
    Note that a connected girth 4 graph must automatically be neighborhood-distinguishable: if $u$ and $v$ have the same closed neighborhood, then they must be adjacent and have a common neighbor $w$ (since otherwise there would be only two vertices), which would give us a triangle. 

    With the theory already developed, girth 4 graphs will be easy to study due to the following simple result. 

    \begin{proposition} \label{prop:c-on-edge}
    A girth 4 graph is positively neighborly.
    \end{proposition}

    \begin{proof}
    This follows from the fact that if $u$ and $v$ are adjacent vertices in a graph of girth 4, the intersection of their closed neighborhoods consists solely of $u$ and $v$.     
    \end{proof}

    \begin{theorem} \label{thm:girth-4-finite-index}
    Suppose $\G$ is a graph of girth 4. 
    \begin{enumerate}
        \item If $\G$ is not bipartite, then it is either RA or $1/2$-RA, and it is $1/2$-RA if and only if every vertex has odd degree and every pair of vertices $2$ apart have an even number of common neighbors.
        \item If $\G$ is bipartite, then $\G$ is $1/\mu$-RA, where $\mu = \gcd(\delta,\kappa)$, with
        \[ \delta = \gcd(\{\deg(v) - 1 : v \in \G\}) \textrm{ and}\]
    \[ \kappa = \gcd(\{|N[u] \cap N[v]| : u,v \in \G, d(u,v) = 2\}).\]
    \end{enumerate}
    \end{theorem}

    \begin{proof}
    This follows directly from \cref{cor:pos-nbr-1/mu} once we note that $N[u] \cap N[v] = N(u) \cap N(v)$ for non-adjacent vertices $u$ and $v$.
    \end{proof}

    In \cite{FirstPaper}, we showed that if $\G$ does not satisfy the ``square completion property'', that every path $u - v - w$ must be able to be completed to a 4-cycle $u - v - w - x - u$, then $\G$ is RA. Not having this property implies that $\kappa = 1$, so \cref{thm:girth-4-finite-index} may be thought of as a generalization of and addition to this fact. In particular, \cref{thm:girth-4-finite-index} implies that if, for example, a girth 4 graph has two vertices with coprime valency, then the graph is RA. It also puts strong restrictions on the structure of a girth 4 graph whose RA matrix $C_\G$ has an elementary divisor of $k$. In particular, such a graph must have the degree of every vertex congruent to $1 \pmod k$, and the number of common neighbors of any vertices a distance 2 apart must be divisible by $k$.

    \begin{example} \label{odd-cube-half-ra}
    The cube graph $Q_d$ is $1/2$-RA for odd $d \geq 3$, and the folded cube graph $\square_d$ is $1/2$-RA for odd $d \geq 5$. Indeed, every vertex of each graph has odd degree, and vertices at a distance of 2 have exactly 2 common neighbors. In fact, as we will see in \cref{thm:deg3notRA=Q3}, the cube graph $Q_3$ is the only non-RA girth 4 graph with vertices of ``small'' degree
    \end{example}


    A natural question to ask at this point is whether there are graphs of girth 4 with an arbitrarily large elementary divisor. Indeed, we can now give a concrete example of a $1/\mu$-RA graph of girth 4 for each $\mu \geq 2$.

    \begin{corollary}\label{cor:crown-graph}
    The crown graph $\Cr(2n)$ on $2n$ vertices (with $n \geq 4$) is a girth 4 graph that is $1/(n-2)$-RA. 
    \end{corollary}

    \begin{proof}
    This follows from \cref{thm:tensor-completes} or \cref{thm:girth-4-finite-index}, since the crown graph is a bipartite graph of girth 4, every vertex has degree $n-1$, and vertices that are 2 apart have $n-2$ common neighbors.
    \end{proof}

    Let us now show that for any $\mu > 1$, any bipartite graph of girth 4 that is $1/\mu$-RA gives rise to many infinite families of such graphs.

    \begin{corollary}
    \label{cor:xK_{k+2}}
    Suppose $\Lambda$ is a bipartite girth 4 graph that is $1/\mu$-RA. Then $\G = \Lambda \times K_{\nu+2}$ is a connected, bipartite girth 4 graph that is $1/\gcd(\mu, \nu)$-RA.
    \end{corollary}
    
    \begin{proof}
    The bipartition of $\Lambda = \Lambda_1 \sqcup \Lambda_2$ induces a bipartition on $\Lambda \times K_{\mu + 2}$ as $(\Lambda_1 \times K_{\mu + 2}) \sqcup (\Lambda_2 \times K_{\mu + 2})$. In general, if $\La$ is any girth 4 graph, its tensor product with another graph has girth 4: the resulting graph $\G$ has no 3-cycles, and any 4-cycle in $\Lambda$ induces a 4-cycle in $\G$ (by just alternating between two neighbors in the other coordinate), so $\G$ has girth 4. Now, $\Lambda$ being $1/\mu$-RA and girth 4 implies, by \cref{thm:girth-4-finite-index}, that $\gcd(\delta_\Lambda, \kappa_\Lambda) = \mu$. Letting $\delta_\G, \kappa_\G$ be the numbers from \cref{thm:tensor-bipartite} and $\delta_\Lambda, \kappa_\Lambda$  from \cref{thm:girth-4-finite-index}, we have
    \[\delta_\G = \gcd\{(\nu + 1)\deg(\lambda) - 1 : \lambda \in \Lambda\}\]
    since $\deg v = \nu + 1$ for all $v \in K_{\nu + 2}$, and $\kappa_\G$ is the $\gcd$ of the two sets 
    \[S_1 = \{\nu \cdot |N(\lambda_1) \cap N(\lambda_2)|: \lambda_1, \lambda_2 \in \Lambda_i, i = 1, 2\},\]
    \[S_2 = \{(\nu + 1)|N(\lambda_1) \cap N(\lambda_2)|: \lambda_1 \ne \lambda_2 \in \Lambda_i, i = 1, 2, d(\lambda_1, \lambda_2) = 2\}\]
    since if we take $u \ne v$ in the expression for $\kappa_\G$, we find $|N(u) \cap N(v)| = \nu$ but if $u = v$ then $|N(u) \cap N(v)| = \nu + 1$ and we cannot have $\lambda_1 = \lambda_2$, which implies that if they have common intersection of their neighborhoods, they must be a distance 2 apart and in the same part of $\Lambda$. Taking $\lambda = \lambda_1 = \lambda_2$ in $S_1$ gives $\nu \deg(\lambda)$ and combining that with $\delta_\G$, we see that $\gcd(\delta_\G, \kappa_\G)$ must divide $\gcd\{\deg(\lambda) - 1 : \lambda \in \Lambda\} = \delta_\Lambda$. Subtracting the expressions in $S_1$ and $S_2$ for $\lambda_1 \ne \lambda_2$, we find $\kappa_\G \mid \gcd\{|N(\lambda_1) \cap N(\lambda_2)|: \lambda_1 \ne \lambda_2 \in \Lambda_i, i = 1, 2, d(\lambda_1, \lambda_2) = 2\}$, thus $\kappa_\G$ must divide $\kappa_\Lambda$. Also, taking any vertex $\lambda \in \Lambda$, we see that $\gcd(\delta_\G, \kappa_\G)$ divides $\nu \deg(\lambda)$ and also $\deg(\lambda) - 1$, hence also $\nu \deg(\lambda) - \nu (\deg(\lambda) - 1) = \nu$. Thus, $\gcd(\delta_\G, \kappa_\G) \mid \gcd(\mu, \nu)$, and it is easy to check the converse holds too 
    so $\gcd(\delta_\G, \kappa_\G) = \gcd(\mu, \nu)$.
    \end{proof}

    \begin{remark}
    \label{rmk:cor-proves-thm}
        \cref{cor:xK_{k+2}} shows that the tensor product $K_2 \times K_{m_2} \times \cdots \times K_{m_n}$ in part (a) of \cref{thm:tensor-completes} is $1/\mu$-RA, where $\mu = \gcd\{m_i - 2 : i = 2, \dots, n\}$ since $K_2 \times K_{m_2}$ is the crown graph $\Cr(2m_2)$, which is $1/(m_2 - 2)$-RA.    
    \end{remark}

    \begin{remark}
    \label{rmk:nonbipartite-girth4-tensor}
        If $\G$ and $\La$ are non-bipartite, then $\G \times \La$ is non-bipartite, 
        and if $\La$ is in addition girth 4 or more, then $\G \times \La$ is a non-bipartite graph of girth 4 (pick a $4$-cycle $\la_1 - \la_2 - \la_3 - \la_4$ in $\La$ and two neighbors $u, v \in \G$; we then have the 4-cycle $(u, \la_1) - (v, \la_2) - (u, \la_3) - (v, \la_4) - (u, \la_1)$), which means we can apply \cref{thm:girth-4-finite-index} to it. That means $\G \times \La$ is RA or 1/2-RA. Since $\deg(v, \la) = \deg(v) \deg(\la)$ for $v\in \G, \la \in \La$ and for $(u, \la_1)$ and $(v, \la_2)$ a distance 2 away we have $|N(u, \la_1) \cap N(v, \la_2)| = |N(u) \cap N(v)| \cdot |N(\la_1) \cap N(\la_2)|$ so $\G \times \La$ is RA unless $\deg v$ and $\deg \la$ are always odd and at least one of $\G, \La$ has every pair of vertices a distance 2 apart sharing an even number of common neighbors. Note that if $\La$ is a graph of girth 5 or higher, its cartesian product with any other connected graph is RA (see \cref{cor:girth-5}), while the same is not always true of tensor products.
    \end{remark}

    \begin{corollary}
    For each $\mu \geq 2$, there are infinitely many bipartite graphs of girth $4$ that are $1/\mu$-RA.
    \end{corollary}
    
    Note that the crown graph on 8 vertices is precisely the cube graph $Q_3$, which is the smallest (connected, neighborhood-distinguishable) non-RA graph. Actually, the crown graphs are precisely the unique smallest $1/k$-RA girth 4 graphs:

    \begin{theorem}\label{thm:crown-graph-smallest}
    For each $k \geq 2$, the crown graph $\Cr(2k + 4)$ is the unique smallest graph $\G$ of girth $4$ that is $1/k$-RA. 
    \end{theorem}

    \begin{proof}
    Suppose $\G$ is a graph of girth 4 that is $1/k$-RA. By \cref{thm:girth-4-finite-index}, the degree of each vertex is congruent to 1 (mod $k$). If $\G$ has at most $2k+4$ vertices, then the degree of each vertex is at most $2k+3$, and thus the degree must be either $k+1$ or $2k+1$, or $3k+1$ if $k = 2$. In this last case, we have $k = 2$ and the vertices have degree 7, so there must be exactly $2k+4 = 8$ vertices; this implies that the graph is $K_8$, which does not have girth 4. So each vertex has degree at most $2k+1$.
    
    Suppose there is a vertex $v$ of degree $2k+1$. Since each neighbor of $v$ has degree at least $k+1$, there must be at least $k$ vertices that are distance 2 from $v$, and this gives us at least $1 + (2k+1) + k = 3k+2$ vertices, and if $k \geq 3$, then this is greater than $2k+4$. If $k = 2$, then we have a graph with exactly $2k+4 = 8$ vertices, where the vertices adjacent to $v$ have degree 3 and are each adjacent to two vertices other than $v$. Since there are only 8 vertices, all neighbors of $v$ must be adjacent to the same two other vertices. However, then both vertices that are distance 2 from $v$ have 5 common neighbors with $v$, whereas \cref{thm:girth-4-finite-index} implies that the number of common neighbors should be even. So if $\G$ has a vertex of degree $2k+1$ and $C_\G$ has an elementary divisor of $k$, then $\G$ must have more than $2k+4$ vertices.

    We may now assume that every vertex has degree $k+1$. Fix a vertex $v$. Its $k+1$ neighbors each have $k$ edges that lead to vertices that are 2 apart from $v$, so there are $k(k+1)$ such edges. Now, by \cref{thm:girth-4-finite-index}, the number of common neighbors between $v$ and vertices that are 2 away is divisible by $k$, and in any case it must be at least 1 and smaller than the degree of $v$. Thus the vertices that are 2 apart from $v$ must have exactly $k$ common neighbors with $v$, and so there must be $k+1$ of them. This gives us $2k+3$ vertices so far, and the vertices that are distance 2 from $v$ are still missing an edge. Clearly then, the only way to add only a single vertex to the graph is to make it adjacent to all vertices that are 2 away from $v$, and what we get is precisely the crown graph.
    \end{proof}

We finish the section with another application of \cref{thm:girth-4-finite-index} to show the cube $Q_3$ is the non-RA graph of girth 4 with the smallest degree of vertices.

\begin{theorem}
\label{thm:deg3notRA=Q3}
    If $\G$ is a connected graph of girth 4 whose vertices have degree $\le 3$ and $\G$ is not RA, then $\G$ is the 3-cube $Q_3$. 
\end{theorem}
\begin{proof}
Since $\G$ has girth 4 and is not RA, \cref{thm:girth-4-finite-index} says that it must be $1/\mu$-RA, where $\mu = \gcd(\delta, \kappa)$ as defined in \cref{thm:girth-4-finite-index}. Then, since the degree of each vertex is at most 3, $\mu = 2$ and the degree of every vertex must be 3. Furthermore, $\kappa = 2$; that is, vertices that are 2 apart have exactly two common neighbors.

Now, pick an initial vertex of $\G$ and label it 000. It must have degree 3, so we’ll call its neighbors 100, 010, and 001 in arbitrary order. Since vertices that are two apart have exactly two common neighbors, there must be another vertex adjacent to 100 and 010; we'll call it 110. Similarly, there must be a vertex 101 adjacent to 100 and 001, and a vertex 011 adjacent to 010 and 001. Note that these three new vertices are all distinct; if any two of them were the same, then that vertex (say 100) would have 3 common neighbors with 000 rather than 2.

Now 000, 100, 010, 001 all have degree 3 with the listed vertices. That means some vertex, which we'll call 111, must complete the 4-cycle $110 - 100 - 101 - 111 - 110$. Note that 111 can’t be 011 because then $011 - 101 - 001$ would be a 3-cycle. To prove that 111 is connected to 011, note we must have a 4-cycle $011 - 010 - 110 - v - 011$, and since 110 already has degree 3, $v$ must be either 111 or 100, but the latter already has degree 3. Now each of the listed vertices has degree 3 so there can’t be any more connections, and there can’t be any more vertices since $\G$ is connected. \end{proof}

\section{Girth 3 graphs can be arbitrarily far from RA}
\label{sec:girth-3}

We have seen that graphs of girth 5 or higher are RA, and graphs of girth 4 are never far off, with $\Rg$ having full rank and at most one nontrivial elementary divisor. In this section, we will show that for girth 3 graphs $\G$, $\Rg$ can have arbitrarily many nontrivial, nonzero elementary divisors and arbitrarily large nullity over $\Z$. In fact, we will show in \cref{thm:allyourbase} that for any numbers $1 < d_1 \mid d_2 \mid \cdots \mid d_m$ and any nullity choice $n$, it is possible to create a graph $\G$ for which $\Rg$ has the $d_i$ as the elementary divisors and nullity $n$ over $\Z$. The distinguishing fact for girth 4 graphs is the fact that adjacent vertices $u$ and $v$ have no common neighbors, so that $\bbe_u + \bbe_v \in \Rg$ for any edge $u-v$. In girth 3 graphs, adjacent vertices may have any number of common neighbors, and so there are no rows of $\Rg$ that are guaranteed to be simple.


We will start by showing that we can find small examples of graphs, namely having only $2k+5$ vertices, that are $1/k$-RA for any $k \ge 2$. In addition to the small examples it provides, the following construction and proof will be useful in proving \cref{thm:allyourbase}. Recall that the join graph $\G_1 + \G_2$ has vertices that are the disjoint union of the vertices of the two graphs $\G_1$ and $\G_2$ and whose edges are the ones from the graphs $\G_i$ together with those joining every vertex of every $\G_1$ to every vertex of $\G_2$. The pyramid over a graph $\G$ is just its join with the one-vertex graph $K_1$. 

\begin{proposition}
\label{prop:pyramid}
The pyramid $\Delta = \Cr(2k+4) + K_1$ over the crown graph on $2k+4$ vertices is $1/k$-RA.
\end{proposition}
\begin{proof}
    Adding the apex vertex $v$ to the crown graph $\Cr(2k+4)$ as the first vertex changes the RA matrix of $C_{\Cr(2k+4)}$ precisely by adding a row of 1's to the top and a column of 1's to the left. To see this, note that the vertex $v$ itself is connected to all vertices, so the top row, corresponding to $N[v]$, is all 1; now for every vertex $u$ the row corresponding to $N[u]$ has a 1 in the leftmost column since $u$ is adjacent to $v$, and furthermore $v \in N[u_1] \cap N[u_2]$ for any two vertices $u_1, u_2$. No more rows are added to $C_\G$ because the intersections $N[u] \cap N[v]$ are just $N[u]$.

    Label the vertices of $\Cr(2k+4)$ as $\{1, 2, \dots, (k+2), 1', 2', \dots, (k+2)'\}$ so the vertex labelled $i$ is adjacent to all $j'$, and $i'$ is adjacent to all $j$, such that $j \ne i$. To find the Smith normal form of $C_\Delta$, we first note that the rows corresponding to $N[1]$ and $N[1']$ only intersect in the first column. Using this row of $C_{\Delta}$, we can thus zero out the entire rest of the first column. Now these rows add up to a row of all-1's (with a 0 in the first column), so subtracting these two rows from the top row makes that row zero. Now, we can work with the submatrix of $C_\Delta$ with the first row and column removed, find its Smith normal form, and just note that we get an extra elementary divisor of 1 from the left out row and column, leaving us with just one nontrivial elementary divisor equal to $k$ by \cref{cor:crown-graph}.
\end{proof}

Using SageMath and data from \cite{mckay_graphs_data_2025}, we found that the pyramid over $\Cr(8) = Q_3$ is one of three 1/2-RA graphs of girth 3 with 9 vertices, and no 1/2-RA girth 3 graphs have fewer vertices. The other two graphs have graph6 strings \verb|H?zTb_{| and \verb|HCOfFz~|. It is unclear whether the pyramid over the crown graph is always one of the minimal girth 3 $(1/k)$-RA graphs; we leave this question as an open problem - see \cref{prob:min-girth3-1/k}.

Unlike with girth 4, we can certainly find graphs $\G$ of girth 3 where $C_\G$ has multiple nontrivial elementary divisors. For example, if $\G$ is the Kneser graph $\Kn(6,2)$ on 15 vertices, then $C_\G$ has elementary divisors $[1^{11}, 2^4]$. In fact, we now show that Kneser graphs $\Kn(kp,p)$ have an unbounded number of nontrivial elementary divisors as the prime $p \to \infty$. We start by exhibiting a family of vectors in the kernel of $C_\G$ for $\G = \Kn(kp, p)$.

\begin{theorem}\label{thm:kneser-very-not-RA}
Suppose $p$ is a prime number and $k \geq 3$. Let $\G = \Kn(kp,p)$. For each vertex $x \in \G$, define the ${kp\choose p}$-dimensional vector $\bx = \sum_{v \in \G} |x \cap v| \bbe_v$ (recalling that $x$ and $v$ are $p$-subsets of $\{1, \ldots, kp\}$). Then $\bx$ lies in the kernel of $C_\G$ over $\Z/p\Z$.
\end{theorem}

\begin{proof}
Without loss of generality, assume $x = \{1,2,\ldots,p\}$. What we need to show is that for each row $\by$ of $C_\G$, $\by \cdot \bx = 0$ (in $\Z/p\Z$). Now, each $\by$ is a $(0,1)$-vector, so it corresponds to a subset $S$ of the vertices. Thus, $\by \cdot \bx$ is simply $\sum_{v \in S} |x \cap v|$, which is to say that it counts the total number of occurrences of numbers $1$ through $p$ among the vertices of $S$. Thus, our goal is to show that for each row of $C_\G$, the total number of occurrences of numbers $1$ through $p$ is a multiple of $p$.

The first type of row corresponds to the closed neighborhood of some vertex $u$. Suppose that $|x \cap u| = r$; without loss of generality we will assume that $u \cap \{1, 2, \ldots, p\} = \{1, 2, \ldots, r\}$. Now, $u$ has ${kp-p \choose p}$ neighbors, and the numbers in $\{1, \ldots, kp\} \setminus u$ each occur the same number of times among the neighbors. Since each vertex is a set of size $p$, and there are $kp-p$ numbers that can occur, we see that the number of times each number appears is
\[{kp - p \choose p} \cdot \frac{p}{kp-p} = \frac{(kp-p-1)(kp-p-2)\cdots (kp-2p+1)}{(p-1)(p-2) \cdots 1} \equiv 1 \pmod p.\]
Since there are $p-r$ numbers $r+1, r+2, \ldots, p$, it follows that the total number of occurrences of these numbers is congruent to $p-r$ (mod $p$). Finally, $u$ itself contains $1, 2, \ldots, r$, contributing another $r$ to the total, so that the total number of occurrences is congruent to $0$ (mod $p$).

The second type of row of $C_\G$ describes the intersection of the neighborhoods of adjacent vertices $u$ and $v$. Without loss of generality, let us assume that $u \cap \{1, \ldots, p\} = \{1, \ldots, j\}$ and $v \cap \{1, \ldots, p\} = \{j+1, \ldots, j+r\}$. Similar to before, $u$ and $v$ have ${kp-2p \choose p}$ common neighbors, with each of the $kp-2p$ possible elements occurring the same number of times. Thus, the number of times each number appears is
\[{kp - 2p \choose p} \cdot \frac{p}{kp-2p} = \frac{(kp-2p-1)(kp-2p-2) \cdots (kp-3p+1)}{(p-1)(p-2) \cdots 1} \equiv 1 \pmod p.\]
Thus, the total number of occurrences of the numbers $j+r+1, \ldots, p$ is congruent to $p-(j+r)$ (mod $p$), and then including the fact that $u$ and $v$ themselves contribute $j+r$ to the total, we again see that the total is divisible by $p$.

Finally, the third type of row describes the intersection of neighborhoods of vertices $u$ and $v$ that are distance 2 apart. Let us assume that $|u \cap v| = m$. Then the number of common neighbors of $u$ and $v$ is ${kp-2p+m \choose p}$, and the number of times each number appears among these vertices is
\[{kp-2p+m \choose p} \cdot \frac{p}{kp - 2p + m} = \frac{(kp-2p+m-1)(kp-2p+m-2)\cdots (kp-3p+m+1)}{(p-1)!} \equiv 0 \pmod p, \]
since $1 \le m \le p-1$ guarantees exactly one factor in the numerator is divisible by $p$ (whereas nothing in the denominator is since $p$ is prime). Thus 
the total number of occurrences of the numbers between $1$ and $p$ is divisible by $p$.
\end{proof}

Now we know some of the contents of the kernel of $C_\G$ for $\G = \Kn(kp, p)$. In the following lemma, we generalize slightly to $\Kn(n, p)$ for any $n > p$ and show that it is not hard to find the number of dimensions that the vectors $\bx$ described in \cref{thm:kneser-very-not-RA} span in general.

\begin{theorem}\label{thm:span-of-CVs}
    The vectors described in \cref{thm:kneser-very-not-RA} for the Kneser graph $\Kn(n, p)$, with $n > p$, span a vector space $\KnV(n, p)$ of dimension $m$, where $m = n - 2$ if $p \mid n$, and $m = n - 1$ otherwise.
\end{theorem}
\begin{proof}
    We define the vector space $\Tup(n, p)$ of all tuples $(a_1, \dots, a_n)$ where each $a_i \in \Z/p\Z$ and $\sum_{i=1}^n a_i \equiv 0 \pmod p$. This is naturally an $(n-1)$-dimensional $\Z/p\Z$-vector space (since the last coordinate must simply ensure $\sum_{i=1}^n a_i \equiv 0 \pmod p$). 
    
    One obvious spanning set for $\Tup(n, p)$ is $\{\bb_i\}_{i=1}^{n-1}$, where $\bb_i$ has $a_i = 1, a_n = -1$, and $a_j = 0$ for $j \notin \{i, n\}$. Another spanning set is the set of vertices of $\Kn(n, p)$, where the vertex $\{i_1, \dots, i_p\}$ is identified with its characteristic tuple, with $a_{i_1} = \cdots = a_{i_p} = 1$ and $a_j = 0$ for $j \notin \{i_1, \dots, i_p\}$. This is because for $i < n$, we have $\bb_i = u - v$, where $u$ is a vertex containing $i$ and $p-1$ other numbers and $v$ contains those same $p-1$ numbers and $n$ instead of $i$. 
    
    There is a surjective linear map $q: \Tup(n, p) \to \KnV(n, p)$ which takes a tuple $t = (a_1, \dots, a_n)$ to the vector $\bv$ defined by declaring its coordinate corresponding to a vertex $u$ to be given by $\sum_{i \in u} a_i \pmod p$ (the image of the map lands in $\KnV(n, p)$ because the image of each vertex $v \in \Kn(n, p)$ does, and the vertices span $\Tup(n, p)$). 
    This shows that we always have $\dim(\KnV(n, p)) \le n-1$, and furthermore when $p \mid n$, the tuple $\vec 1 = (1, \dots, 1)$ is in the kernel
    , so in this case the dimension is reduced by one and we have $\dim(\KnV(n, p)) \le n-2$.
    
    Finally, let us show that there is no kernel if $p \nmid n$ and that $\langle \vec 1 \rangle$ is the entire kernel when $p \mid n$. Suppose that $t = (a_1, \dots, a_n) \in \ker q$. Considering the entries of $q(t)$ corresponding to $\{1, \dots, p\}$ and $\{1, \dots, p-1, p+1\}$ shows that $a_p \equiv a_{p+1} \pmod p$. Similarly we see that $a_i \equiv a_j \pmod p$ for all $i \ne j$, so only a multiple of $\vec 1$ can be in the kernel, so $\vec 1 \in \Tup(n, p)$ precisely when $p \mid n$.
\end{proof}

\begin{remark}
    One can show that a basis for $\KnV(n, p)$ is given by $\{\bv_p, \bv_{p+1}, \dots, \bv_{m+1}, \bv_{\widehat{p-1}}, \bv_{\widehat{p-2}}, \dots, \bv_{\widehat2}\}$, the vectors corresponding to vertices $v_i = \{1, \dots, p-1, i\}$, $v_{\widehat{j}} = \{1, \dots, \widehat{j}, \dots, p+1\}$ (the numbers $1, \dots, p+1$ with $j$ missing).
\end{remark} 

The preceding two lemmas show that we can find  girth-3 graphs with an arbitrarily large number of elementary divisors divisible by any fixed prime $p$:

\begin{corollary}
For each prime $p$ and each positive integer $m$, there is a graph $\G$ such that $C_\G$ has at least $m$ elementary divisors that are multiples of $p$.
\end{corollary}

Note that we are not describing the entire kernel of $C_\G$ in Theorems \ref{thm:kneser-very-not-RA} and \ref{thm:span-of-CVs}. For example, if $\G = \Kn(8,2)$, then we can compute that the elementary divisors of $C_\G$ are $[1^{21}, 2^7]$, so we have only accounted for 1/2 of the vectors in the kernel. The following table shows the nontrivial elementary divisors of $C_\G$ for some small Kneser graphs $\Kn(n, p)$ for multiples $n$ of $p$.

\begin{table}[hbtp]
    \centering
    \begin{tabular}{|c|c|c|c|c|c|c|c|c|c|c|c|}
       \hline
       $(n, p)$ & (6, 2) & (8, 2) & (10, 2) & (12, 2) & (14, 2) & (16, 2) & (18, 2) & (20, 2) & (9, 3) & (12, 3) & (15, 3) \\
       \hline
       $C_{\Kn(n, p)}$: & $2^4$ & $2^7$ & $2^8$ & $2^{10}, 4^1$ & $2^{12}$ & $2^{15}$ & $2^{16}$ & $2^{18}, 4^1$ & $3^7$ & $3^{10}$ & $3^1, 9^{13}$\\
       \hline
    \end{tabular}
    \caption{The elementary divisors $> 1$ of $C_\G$ for the Kneser graphs $\Kn(kp, p)$}
    \label{tab:KneserElDiv}
\end{table}

Girth 3 graphs can in fact get still much farther from being RA and the $C_\G$ matrix can have arbitrarily many elementary divisors that are 0; that is, its right kernel can have arbitrarily large dimension, as we proceed to show now.     For any integer $n \ge 2$, let $r = \lceil \log_2 n \rceil$ and $M = M(n)$ be the $(n + r) \times (n + r)$ block matrix\footnote{Idea adapted from ChatGPT (OpenAI), conversation with Igor Minevich, August 16, 2025.}
    \[\mat{J_n}B{B^T}{J_r}\]
where $J_i$ is the all-ones $i \times i$ matrix and $B$ is the $n \times r$ matrix whose $k$-th row (consisting of 0's and 1's) is the number $k-1$ written out in binary. Since $M$ is symmetric and has 1's along the main diagonal, $M-I$ is the adjacency matrix for a graph $\Bg(n)$ ($\Bg$ stands for ``binary graph'').

We can also think of $\Bg(n)$ in terms of an incidence graph. Consider a bipartite graph where one part consists of vertices $1, 2, \ldots, r$ and the other part consists of the first $n$ subsets of $\{1, 2, \ldots, r\}$ (in an ordering where $A < B$ if $\max(A) \leq \max(B)$ and $A$ is lexicographically before $B$). Then put an edge between each number and the sets that contain it. Finally, we convert each part of the bipartition into a clique.

\begin{theorem}\label{thm:z(n)}
Let $z(n)$ be the dimension of the right kernel of $C_{\Bg(n)}$. Then, $z(n+1) = z(n) + 1$ if $n$ has at least three 1's in its binary representation, and otherwise $z(n+1) = z(n)$.
\end{theorem}
\begin{proof}
    First we show that any vector $\bv$ in the kernel must be zero in the last $r$ coordinates. Row 1 has 1's in only the first $n$ places, so the sum of the first $n$ entries of $\bv$ must be 0. But rows $2, 3, 5, 9, \dots, 2^{i-1}+1, \dots, 2^{r-1}+1$ have 1's in the first $n$ entries, a single 1 in the $i$-th place, and 0's elsewhere in the last $r$ entries, for $i = 1, 2, \dots, r$. This forces the $i$-th coordinate of any vector in the kernel to be 0 (for $n+1 \leq i \leq n+r$). 

    Now, a vector $\bv$ in the kernel of $C_{\Bg(n)}$ must simply have the sum of its elements equal to 0 and dot product 0 with every intersection of two (not necessarily distinct) rows of $B^T$ (where intersection is the bitwise AND operation). 

    If a vector $\bv$ is in the kernel of $C_{\Bg(n)}$, then extending it by a 0 on the right gives a vector in the kernel of $C_{\Bg(n+1)}$. This is because the entries in the first $n$ columns of $B^T$ do not change as $n$ increases, as more columns are added for the binary representations of further rows and eventually more rows are added to $B^T$ that have 0's in the first $n$ entries, so if the dot product of the first $n$ entries of $\bv$ with the intersections of rows of $B^T$ was 0 in $M(n)$, it will remain zero for $B^T$ in $M(n+1)$.

    If $n$ has only one 1 in its binary representation, so will the $(n+1)$st column of $B^T$ in $M(n+1)$. The rightmost column of $B^T$ obviously always has a 1 at the top since a nonzero binary number starts with 1, so this 1 is in the top row. In this case, $n$ is a power of 2 and no previous number has been as large as that power of 2, so there are only 0's to the left of this 1. Any vector $\bv$ in the kernel must have a zero dot product with this row, which means its entry in that $(n+1)$st column must be 0 and it must come from adding 0's to the right of vectors that were in the kernel for $C_{\Bg(n)}$.
    
    If $n$ has only two 1's in its binary representation, then one of the 1's is in the top row and by the time the 1's begin to appear in the top row, there are only 0's to the left of the bottom 1, 
    so the two rows corresponding to the two 1's only intersect in the $(n+1)$st column. (To be explicit, this is because $n$ must equal {$2^{r-1} + 2^i$} for some $i, 0 \le i \le r-1$, the only ones in the top ($r$-th) row of $B^T$ are the ones in columns $2^{r-1} + 1$ through $2^{r-1} + 2^i$, and none of the numbers between $2^r$ and $2^r + 2^i-1$ (inclusive) have a 1 in the $i$-th binary digit.) Again, any vector $\bv$ in the kernel must have dot product 0 with this intersection, so its $(n+1)$st entry must be 0 and it comes from the kernel for $C_{\Bg(n)}$. This proves that if $n$ has less than three 1's in its binary representation, then $z(n+1) = z(n)$.

    There is another way to see this, relying on the same reasoning. We will restrict our attention only to the first $n$ coordinates of rows in $C_{\Bg(n)}$ for all $n$, as these are really responsible for the right kernel of $C_{\Bg(n)}$; we will denote by $C'_{\Bg(n)}$ the first $n$ columns of $C_{\Bg(n)}$. As it turns out, the number of distinct nonzero rows in $C'_{\Bg(n)}$ increases by 1 if $n$ has one or two 1's in its binary representation and it does not increase if it has at least three; what's more, these are all linearly independent. If $n$ has only one 1 in its binary representation, then $C'_{\Bg(n+1)}$ gains a new row in $B^T$ which is just 1 in the $(n+1)$st column and zeroes elsewhere; this has trivial intersection with all other rows in $C'_{\Bg(n+1)}$ so only one vector is added, and is indeed linearly independent to those that came before. If $n$ has two 1's in its binary representation, then rows in $C'_{\Bg(n)}$ are extended to longer rows in $C'_{\Bg(n+1)}$, and if they were already linearly independent, they remain so. Now by the reasoning above, $C_{\Bg(n+1)}$ has a new row with a 1 in the $(n+1)$st column and zeroes elsewhere; again, this is clearly independent of any vector whose nonzero entries were confined to the first $n$ columns. These are linearly independent, so by the rank-nullity theorem we have $z(n+1) = z(n)$ if $n$ has one or two 1's in its binary representation.

    If $n$ has at least three 1's in its binary representation, then there are 1's to the left of all of these 1's in the $(n+1)$st column, and furthermore we shall show that there is no intersection of these rows that yields a 1 only in the $(n+1)$st column. Suppose the 1's in the $(n+1)$st column occur in rows $i_1, i_2, \dots, i_m, r$ of $B^T$ (in order from the bottom up). Any two rows $i_a$ and $i_b$ with $i_a < i_b < r$ intersect at least once again in the $(n+1 - 2^{r-1})$-th column. Indeed, the pattern of the first $2^{r-1}$ columns and first $r-1$ rows repeats starting at the $(2^{r-1}+1)$st column. For the same reason, for each $a < m$, the top row intersects the $i_a$-th row at least once again in the $(n+1 - 2^{i_m})$-th column.
    Finally, the top row intersects with row $i_m$ in the $n$-th column since the existence of a third 1 in the $(n+1)$st column guarantees $n$ is not the first number to have a 1 in the $i_m$-th digit after the numbers have started having a 1 in the $r$-th digit. 
    
    Thus, no row of $C_{\Bg(n+1)}$ has a single 1 in the $(n+1)$st column and zeroes elsewhere. Every intersection of two rows of $B^T$ in $M(n+1)$ thus already has nonzero entries in the first $n$ columns so comes from a row in $C_{\Bg(n)}$ and adds no new rows to $C_{\Bg(n+1)}$. Since $n$ has increased by exactly 1 and the kernel depends only on $n$, not $r$, the rank-nullity theorem says that the dimension of the right kernel of $C_{\Bg(n+1)}$ is one more than that of $C_{\Bg(n)}$.    
\end{proof}

Since infinitely many numbers have at least three 1's in their binary representation, we have the following theorem.

\begin{corollary}
 \label{cor:unbounded-zero-el-div}
    As $n \to \infty$, the number of zero elementary divisors $z(n)$ of $C_{\Bg(n)}$ increases without bound, and the function $n \mapsto z(n)$ is surjective onto the set of nonnegative integers. In particular, the dimension of the kernel is
    \[ z(n) = n - 1 - r - {r-1 \choose 2} - \lceil \log_2(n-2^{r-1}) \rceil.\]
\end{corollary}

\begin{proof}
Of course, $\lim_{n \to \infty} z(n) = \infty$ since there are infinitely many numbers with at least 3 1's in binary. The function $n \mapsto z(n)$ is surjective because $z(2) = 0$ and as $n$ goes up by 1, $z(n)$ goes up by at most 1. The value of $z(n)$ is simply how many positive integers less than $n$ have at least three 1's in their binary representation. Starting with the $n-1$ numbers $1$ through $n-1$, we have exactly $r$ of them that have a single 1 in their binary expansion: $1, 2, 4, \ldots, 2^{r-1}$. There are ${r \choose 2}$ numbers less than $2^r$ with exactly two 1's in their binary expansion, but some of these may not be less than $n$. Certainly, since $n > 2^{r-1}$, every way of having two 1's that does not use the leftmost bit will be less than $n$, so that gives us ${r-1 \choose 2}$ numbers. For the remainder, we simply subtract off the leftmost bit from $n$ by computing $n - 2^{r-1}$, and then $\lceil \log_2(n - 2^{r-1}) \rceil$ gives us how many numbers less than $n - 2^{r-1}$ have a single 1 in their binary expansion.
\end{proof}

We note that the smallest graphs with a zero elementary divisor of $C_\G$ have 10 vertices: they were found by exhaustive search using SageMath and the data in \cite{mckay_graphs_data_2025}. They have graph6 strings \verb|I?otQji\O| and \verb|ICQrThix_|, and both have $C_\G$ with elementary divisors $1^9, 0$. The smallest $n$ such that $C_{\Bg(n)}$ has nontrivial kernel is $n = 8$, giving a graph with 11 vertices.

We have shown in \cref{cor:crown-graph} that $C_\G$ can have any single elementary divisor we wish, and in \cref{cor:unbounded-zero-el-div} that we can make $C_\G$ have any nullity. Let us now use the crown graphs and binary graphs as building blocks to show that we can make $C_\G$ as ``bad'' as we want.

\begin{theorem}\label{thm:allyourbase}
    Given any list of numbers $d_1 \mid d_2 \mid \cdots \mid d_n$ with $d_1 \ge 2$, and any number $r \ge 0$, there is a graph $\G$ whose RA matrix $C_\G$ has nullity $r$ and nontrivial, nonzero elementary divisors precisely $d_1, \dots, d_n$.
\end{theorem}
\begin{proof}
    Let $\G_i = \Cr(2d_i + 4)$ be the crown graph on $2d_i + 4$ vertices, and (if $r > 0)$ let $N$ be an integer such that $z(N) = r$ (from \cref{cor:unbounded-zero-el-div}) and $\Bg = \Bg(N)$. Our graph $\G$ is $K_1 + \bigsqcup \G_i \sqcup \Bg$, the pyramid over the disjoint union of all of these (if $r = 0$, we do not include $\Bg$).
    Let $A_i$ be the activation matrix of $\Cr(2d_i + 4)$ and $A$ the activation matrix for $\Bg$. Then, ordering the vertices with the pyramid's apex first, the activation matrix $A_\G$ of $\G$ takes the form
    \[\begin{pmatrix}
        1 & 1 \cdots 1 & 1 \cdots 1 & \cdots & 1 \cdots 1\\
        \vec 1 & A_1 & 0 & \cdots & 0\\
        \vec 1 & 0 & A_2 & \cdots & 0\\
        \vdots & \vdots & \vdots & \ddots & \vdots\\
        \vec 1 & 0 & 0 & \cdots & A\\        
    \end{pmatrix}
    \]
    The intersections of rows are either (a) the intersection of the first row with a row that comes from some $A_i$ or $A$, which is just the same row as in $A_\G$, (b) the intersection of two rows in the same block, which form a row of $C_{\G_i}$ or $C_{\Bg}$ and a 1 in the leftmost column, or (c) the intersection of two rows from two different blocks, which just form a 1 in the leftmost column.

    Note that, as in the proof of \cref{prop:pyramid}, for every $i$ there are two rows of $A_i$ that have no intersection and add up to the all-ones row of size $2d_i + 4$. The intersection of these two rows in $A_1$ gives just a 1 in the leftmost column and 0's elsewhere, so we can use row operations to zero out the entire left column (excluding one row equal to $\vec e_1^T$) and then the top row above each $A_i$. In the activation matrix $A$ for $\Bg(N)$, we have two particular rows, say row $x$ which looks like $1 \cdots 1 1 0 \cdots 0$ with $N+1$ repeated 1's and row $y$ which looks like $1 \cdots 1 0 1 \cdots 1$, with $N$ 1's followed by a 0 and the rest 1's. Their intersection is just a row with $N$ 1's followed by 0's, so subtracting rows $x$ and $y$ from the first row and then adding their intersection cancels out the first row completely. Now we are just left with a single 1 in the first column and blocks of the form $C_{\G_i}$ and $C_{\Bg}$, so using the column and row operations for each of these gives us diagonals with a single nonzero elementary divisor $d_i$ for each $i = 1, \dots, n$ and $r$ zeros in the diagonal corresponding to $\Bg$. Since each $d_i \mid d_{i+1}$ already by assumption, these are precisely the nonzero elementary divisors of $C_{\G}$.
\end{proof}

Using similar techniques, one can also show that if we build a pyramid over each $\Cr(2d_i + 4)$ and $\Bg(N)$, and then take the disjoint union of all these, adding a single vertex connected to just the apices of these pyramids, this too gives a graph $\G$ with the same elementary divisors of $C_\G$. In fact, so does simply the join $\Cr(2d_1+4) + \Cr(2d_2 + 4) + \cdots + \Cr(2d_n+4) + \Bg(N)$. 



Note that \cref{thm:allyourbase} certainly does not give the smallest graph satisfying the desired properties. For example, taking the join of four crown graphs $\Cr(8)$ gives a 32-vertex graph $\G$ whose $\Rg$ has four elementary divisors equal to 2, but \cref{thm:span-of-CVs} says the Kneser graph $\Kn(6, 2)$ on just 15 vertices already has four elementary divisors of $\Rg$ divisible by 2, and the computations in \cref{tab:KneserElDiv} showed that indeed they are all exactly 2 and there are no more.

\section{Further insights from data}
\label{sec:data}
We collect here some information about small connected graphs; see \cref{tab:8-vertex-RA}. Graphs of girth 5 or more are always RA by \cref{cor:girth-5}, as are graphs on 7 or fewer vertices, as discussed in \cite{FirstPaper}. Graphs of girth 4 are always neighborhood-distinguishable and may be RA or $1/k$-RA for some $k$. Graphs of girth 3 may have vertices with identical neighborhoods, and if not, may be RA or not. Looking at connected graphs with up to 10 vertices, the proportion of graphs that are neighborhood-distinguishable rapidly increases, from about $67\%$ of 8-vertex graphs to about $84\%$ of 10-vertex graphs. The vast majority of small neighborhood-distinguishable graphs are RA, with about $99.9997\%$ of neighborhood-distinguishable graphs on 10 vertices being RA.

The data suggest that being RA is extremely common. However, it is possible that this is just an artifact of dealing with small graphs. In fact, the proportion of graphs on $n$ vertices that are girth $3$ tends to 1, and girth 3 is precisely where there seems to be plenty of opportunity to find obstructions to being RA. We leave the question of the proportion of graphs that are RA as an open problem - see \cref{prob:prop-nonRA}.

\begin{table}[htbp]
\centering
  \begin{tabular}{@{} l l r r r @{}}
    \toprule
    Girth      & Category                                  & 8 vertices & 9 vertices & 10 vertices \\ 
    \midrule
    \multirow{3}{*}{3}
               & Nbhd-Indist.            & 3675 & 63308 &  1908362\\
               & Nbhd-Dist. RA           & 7175 & 196389 & 9798347 \\
               & Nbhd-Dist. not RA       & 0 & 3 & 30   \\
    \addlinespace
    \multirow{2}{*}{4}
               & RA                                        & 219 & 1243 & 9367 \\
               & not RA                                    & 1 & 0 & 1   \\
    \addlinespace
    5+    & (all RA)                                  & 47 & 137 & 464 \\
    \midrule
    \multicolumn{2}{r}{Total}                                   & 11117 & 261080 & 11716571 \\
    \bottomrule
  \end{tabular}
  \caption{Counts of connected graphs on a fixed number of vertices by girth, category, and RA status.}
  \label{tab:8-vertex-RA}
\end{table}


\section{Conclusion and Open Problems}

We have fairly comprehensively explored the possibilities for the elementary divisors of the RA matrix $C_\G$ of a graph $\G$. Here we collect a variety of open problems for future research.

In our explorations of products of graphs, we were able to characterize the elementary divisors of the RA matrix of strong products and cartesian products. For tensor products of graphs, our characterization is complete except for the tensor product of graphs of girth 3.

\begin{problem}
    \label{prob:nonbip-tensor}
    Describe the elementary divisors of $C_\G$ for the tensor product $\G$ of two 
    graphs of girth 3.
\end{problem}

In \cref{thm:crown-graph-smallest}, we showed that the crown graphs yield the smallest $1/k$-RA graphs of girth 4, and \cref{prop:pyramid} showed that the pyramids over crown graphs are $1/k$-RA graphs of girth 3. It is not clear whether the pyramids over crown graphs are always among the minimal girth 3 $1/k$-RA graphs.

\begin{problem}
\label{prob:min-girth3-1/k}
For each $k \geq 2$, determine the smallest graph(s) of girth 3 that are $1/k$-RA.
\end{problem}

The data in \cref{tab:8-vertex-RA}, though limited in scope, suggests that almost every graph is RA. Furthermore, the examples of non-RA graphs we have found so far have been highly structured, providing further evidence that a random graph is almost surely RA.

\begin{problem}
\label{prob:prop-nonRA}
Does the proportion of $n$-vertex graphs that are RA tend toward $1$ as $n \to \infty$?
\end{problem}

When $\G$ is RA, the structure of $G^\G$ essentially boils down to the structure of $(G^\Ab)^\G$ (which itself can be obtained from $\Z^\G$). Explicitly, in this case, the full preimage of each element of $(G^\Ab)^\G$ is contained in $G^\G$. If $\G$ is not RA, then this is no longer the case, though perhaps a similar description of $G^\G$ is possible if $C_\G$ has few elementary divisors. Since we have shown that many graphs are $1/\mu$-RA for some $\mu$, it would be useful to more concretely describe $G^\G$ for such graphs.

\begin{problem}
Describe $G^\G$ in terms of $(G^\Ab)^\G$ in the case where $\G$ is $1/\mu$-RA.
\end{problem}

\cref{thm:allyourbase} shows a simple construction for building graphs where $C_\G$ has any desired nullity and any sequence of nontrivial elementary divisors. Knowing the elementary divisors of $C_\G$ tells us everything we need to know about $[G,G]^{|\G|} \cap G^\G$ if $G$ is a Heisenberg group $H(\F_p)$. For other groups $G$, however, the picture is much more complicated. Recall the chain of subgroups
\[\Comm(G, \G) \le [G^\G, G^\G] \le [G, G]^{|\G|} \cap G^{\G} \le [G, G]^{|\G|}\]
where $\Comm(G, \G)$ is generated by commutators $[g^v, h^w]$ for $g,h \in G$ and vertices $v, w$ in $\G$ (possibly $v = w$).
The group we are most interested in is $[G, G]^{|\G|} \cap G^\G$, so that we can use the short exact sequence
\[1 \to [G, G]^{|\G|} \cap G^\G \to G^\G \to (G/[G, G])^{\G} \to 1\]
to understand $G^\G$. What the RA matrix $C_\G$ tells us is the structure of $\Comm(G,\G)$; in fact, $\Comm(G, \G) = [G, G]^{C_\G}$. Thus, understanding $\Comm(G,\G)$ from the elementary divisors of $C_\G$ is a special case of understanding $G^M$ from the elementary divisors of an arbitrary integer matrix $M$. Certainly the elementary divisors do not tell the whole story; for example, if $G = D_8$, $M_1 = \mat1004$, and $M_2 = \mat1204$, then $M_1$ and $M_2$ have the same elementary divisors, but $G^{M_1} \cong D_8$ while $G^{M_2} \cong C_2 \times D_8$. 

\begin{problem}
Determine to what extent we can describe $G^M$ knowing only the elementary divisors of $M$.
\end{problem}





For Heisenberg groups, it turns out that $\Comm(G,\G) = [G^\G, G^\G] = [G,G]^{|\G|} \cap G^\G$ (see \cite[Theorem 6.10]{FirstPaper}). More generally, any group $G$ with a central commutator subgroup and with \emph{faithful abelian generators} (see \cite[Definition 6.7]{FirstPaper}) will also satisfy this property. What other groups have this property? Is there a family of graphs that can be used to determine whether a group has this property?

\begin{problem}
Characterize groups $G$ such that $\Comm(G,\G) = [G, G]^{|\G|} \cap G^\G$ for all graphs $\G$. 
\end{problem}

\begin{problem}
Is there a family of graphs $\mathcal F$ such that, for a fixed group $G$, $\Comm(G,\G) = [G,G]^{|\G|} \cap G^\G$ for all graphs $\G$ if and only if $\Comm(G,\Lambda) = [G,G]^{|\Lambda|} \cap G^\La$ for some $\La \in \mathcal F$ (or for all $\La \in \mathcal F$)?
\end{problem}

Attacking the same problem from a different direction, if we fix a graph $\G$ such that $C_\G$ has many nontrivial elementary divisors, can we find a group $G$ that makes the gap between $\Comm(G, \G)$ and $[G,G]^{|\G|} \cap G^\G$ as large as possible?

\begin{problem}
For each graph $\G$ such that $C_\G$ has $k$ nontrivial elementary divisors, is it possible to find a group $G$ such that $\Comm(G, \G) \cong [G,G]^{|\G|-k}$ but $\G$ is $G$-RA (that is, $[G,G]^{|\G|} \leq G^\G$)?    
\end{problem}

We have seen that the Heisenberg group $H(\F_p)$ acts as a witness of the number of elementary divisors of $C_\G$ that are divisible by $p$. Is there a group that would serve this same function for prime powers?

\begin{problem}
For each prime power $p^k$, is there a group $G$ whose commutator subgroup is cyclic of order $p^k$ and such that $\Comm(G,\G) = [G, G]^{|\G|} \cap G^\G$ for every $\G$?
\end{problem}



We plan to explore these and related problems in subsequent papers.


\section{Acknowledgments}

The authors used SageMath \cite{sage} and GAP \cite{GAP} for testing hypotheses, and House of Graphs \cite{HouseOfGraphs} for visualizing and searching for graphs. They also had several helpful conversations with ChatGPT. The code used is available upon request.

\bibliographystyle{amsplain}
\bibliography{refs}
\end{document}